\documentclass[a4paper,12pt,leqno]{article}

\title{The Expressive Power of $k$-ary Exclusion Logic}
\author{Raine Rönnholm \\ \small Tampere University
}
\date{}

\usepackage[top=1.05in, bottom=1.05in, left=1.35in, right=1.35in]{geometry}

\usepackage[T1]{fontenc}
\usepackage{lmodern}

\usepackage{amsfonts}
\usepackage{ifthen}
\usepackage{xspace}
\usepackage{xcolor}
\usepackage{lmodern}
\usepackage{comment}
\usepackage{multicol}
\usepackage{adjustbox}

\usepackage{amsmath}
\usepackage{amssymb}
\usepackage{mathtools}
\usepackage{stmaryrd}

\usepackage{amsthm}

\usepackage{enumitem}
\usepackage{tikz}

\theoremstyle{plain}
\newtheorem{theorem}{Theorem}[section] 
\newtheorem{lemma}[theorem]{Lemma}
\newtheorem{corollary}[theorem]{Corollary}
\newtheorem{proposition}[theorem]{Proposition}
\newtheorem{claim}{Claim}

\theoremstyle{definition}
\newtheorem{definition}{Definition}[section] 
\newtheorem{example}{Example}[section] 

\newtheorem*{remark}{Remark}

\allowdisplaybreaks[1]

\newcommand{\tuple}[1]{\vec{#1}} 

\DeclareMathOperator{\true}{\vDash}	
\DeclareMathOperator{\ntrue}{\nvDash} 
\DeclareMathOperator{\dom}{dom} 

\DeclareMathOperator{\vr}{Vr} 
\DeclareMathOperator{\fr}{Fr} 
\DeclareMathOperator{\subf}{Sf}  
\DeclareMathOperator{\Tset}{T_\text{$L$}} 
\DeclareMathOperator{\FOset}{FO_\text{$L$}} 
\DeclareMathOperator{\INCset}{INC_\text{$L$}} 
\DeclareMathOperator{\EXCset}{EXC_\text{$L$}} 
\DeclareMathOperator{\INEXset}{INEX_\text{$L$}} 
\DeclareMathOperator{\ESOset}{ESO_\text{$L$}} 

\DeclareMathOperator{\FO}{FO}
\DeclareMathOperator{\ESO}{ESO}

\DeclareMathOperator{\INC}{INC}
\DeclareMathOperator{\EXC}{EXC}
\DeclareMathOperator{\INEX}{INEX}

\DeclareMathOperator{\inc}{\text{$\subseteq$}} 
\DeclareMathOperator{\exc}{\text{$\mid$}} 
\DeclareMathOperator{\ince}{\text{$\subseteq^{e}$}} 
\DeclareMathOperator{\dep}{\text{$=$}} 

\DeclareMathOperator{\exclusion}{exc}

\DeclareMathOperator{\Ee}{\exists} 
\DeclareMathOperator{\Eu}{\exists^{\bf U}} 
\DeclareMathOperator{\Ae}{\forall} 

\DeclareMathOperator{\uniform}{{\bf U}} 

\newcommand*{\abs}[1]{\left\lvert#1\right\rvert}   

\usepackage{hyperref}

\begin{document}
\maketitle


\begin{abstract}
In this paper we study the expressive power of $k$-ary exclusion logic, EXC[$k$], that is obtained by extending first order logic with $k$-ary exclusion atoms. It is known that without arity bounds exclusion logic is equivalent with dependence logic.  By observing the translations, we see that the expressive power of EXC[$k$] lies in between $k$-ary and ($k\!+\!1$)-ary dependence logics. We will show that, at least in the case when $k=1$, both of these inclusions are proper.

In a recent work by the author it was shown that $k$-ary inclusion-exclusion logic is equivalent with $k$-ary existential second order logic, ESO[$k$]. We will show that, on the level of sentences, it is possible to simulate inclusion atoms with exclusion atoms, and in this way express ESO[$k$]-sentences by using only $k$-ary exclusion atoms. For this translation we also need to introduce a novel method for ``unifying'' the values of certain variables in a team. As a consequence, EXC[$k$] captures ESO[$k$] on the level of sentences, and we obtain a strict arity hierarchy for exclusion logic. It also follows that $k$-ary inclusion logic is strictly weaker than EXC[$k$].
Finally we use similar techniques to formulate a translation from ESO[$k$] to $k$-ary inclusion logic with an alternative strict semantics. Consequently, for any arity fragment of inclusion logic, strict semantics is strictly more expressive than lax semantics.

\medskip
Keywords:
exclusion logic, inclusion logic, dependence logic, team semantics, existential second order logic, expressive power.
\end{abstract}


\section{Introduction}

\emph{Exclusion logic} is an extension of first order logic with \emph{team semantics}. In team semantics the truth of formulas is interpreted by using sets of assignments which are called \emph{teams}. This approach was introduced by Hodges \cite{Hodges97} to define compositional semantics for the \emph{IF-logic} by Hintikka and Sandu \cite{Hintikka89}. The truth for the IF-logic was originally defined by using semantic games of imperfect information (\cite{Hintikka97}), and in thus teams can be seen as sets of parallel positions in a semantic game. Teams can also be interpreted as databases (\cite{Vaananen07}), and thus the study of logics with team semantics has natural connections with the study of database dependencies.

For first order logic team semantics is just a generalization of Tarski semantics and has the same expressive power. But if we extend first order logic with new atomic formulas, we obtain higher expressive power and we can define more complex properties of teams. The first new atoms for this framework were \emph{dependence atoms} introduced by Väänänen \cite{Vaananen07}. In \emph{dependence logic} the semantics for these atoms are defined by functional dependencies of the values of variables in a team. Several new atoms have been presented for this framework with the motivation from simple database dependencies -- such as \emph{independence atoms} by Grädel and Väänänen \cite{Gradel12} and \emph{inclusion and exclusion atoms} by Galliani \cite{Galliani12b}. Lately there has been research on these atoms with an attempt to formalize the dependency phenomena in different fields of science, such as database theory (\cite{Kontinen13b}), belief presentation~(\cite{Galliani12a}) and quantum mechanics (\cite{Paolini15}).

If we extend first order logic with inclusion/exclusion atoms we obtain \emph{inclusion and exclusion logics}. The team semantics for these atoms is very simple: Suppose that $\tuple t_1,\tuple t_2$ are $k$-tuples of terms and $X$ is a team. The $k$-ary inclusion atom $\tuple t_1\!\subseteq\!\tuple t_2$ says that the values of $\tuple t_1$ are included in the values of $\tuple t_2$ in the team $X$. The $k$-ary exclusion atom $\tuple t_1\,|\,\tuple t_2$ dually says that $\tuple t_1$ and $\tuple t_2$ get distinct values in $X$, i.e. for all assignments $s,s'\in X$ we have $s(\tuple t_1)\neq s'(\tuple t_2)$. 


Galliani \cite{Galliani12b} has shown that without arity bounds exclusion logic is equivalent with dependence logic. Thus, on the level of sentences, it captures \emph{existential second order logic}, ESO (\cite{Vaananen07}). Inclusion logic is not comparable with dependence logic in general (\cite{Galliani12b}), but captures \emph{positive greatest fixed point logic} on the level of sentences, as shown by Galliani and Hella~\cite{Hella13}. Hence exclusion logic captures NP and inclusion logic captures PTIME over finite structures with linear order.  

In order to understand the nature of these atoms, there has been research on the bounded arity fragments of the corresponding logics. Durand and Kontinen \cite{Durand12} have shown that, on the level of sentences, $k$-ary dependence logic captures the fragment of  ESO in which at most ($k$$-$$1$)-ary functions can be quantified.\footnote{See \cite{Kontinen13a}, \cite{Hannula14} and \cite{Hannula15} for similar hierarchy results on independence and inclusion logics.}  From this it follows that dependence logic has a strict arity hierarchy over sentences since the arity hierarchy of ESO (over arbitrary vocabulary) is known to be strict, as shown by Ajtai \cite{Ajtai83}. However, these earlier results do not tell much about the expressive power of $k$-ary exclusion logic, EXC[$k$], as the existing translation from it to dependence logic does not respect the arities of atoms. 

There has not been much research on exclusion logic after Galliani proved its equivalence to dependence logic. In this paper we will show that the relationship between these two  logics becomes nontrivial when we consider their bounded arity fragments. This also leads to results on the relation between inclusion and exclusion logics, which is interesting because they can be seen as duals to each other, as we have argued in \cite{Ronnholm15}.

By inspecting Galliani's translations (\cite{Galliani12b}) between exclusion and dependence logics more closely, we observe that EXC[$k$] is stronger than $k$-ary dependence logic but weaker than ($k$$+$$1$)-ary dependence logic. Thus it is natural to ask whether the expressive power of EXC[$k$] is strictly in between $k$-ary and ($k$$+$$1$)-ary dependence logics. We will show that this holds at least when $k=1$.

In an earlier work by the author \cite{Ronnholm15} it was shown that both INC[$k$]- and EXC[$k$]-formulas could be translated into $k$-ary ESO, ESO[$k$], which gives us an upper bound for the expressive power of EXC[$k$]. In \cite{Ronnholm15} it was also shown that conversely ESO[$k$]-formulas with at most $k$-ary free relation variables can be expressed in $k$-ary inclusion-exclusion logic, INEX[$k$], and consequently INEX[$k$] captures ESO[$k$] on the level of sentences. 

Since exclusion logic is closed downwards, unlike inclusion-exclusion logic, we know that EXC[$k$] is strictly weaker than INEX[$k$]. However, in certain cases we can simulate the use of inclusion atoms with exclusion atoms: Suppose that $x$, $w$, $w^c$ are variables such that the sets of values of $w$ and $w^c$ in $X$ are complements of each other. Now we have $\mathcal{M}\true_X x\subseteq w$ iff $\mathcal{M}\true_X x\mid w^c$. This can be generalized for $k$-ary atoms if the values of $k$-tuples $\tuple w$ and $\tuple w^c$ are complementary (with respect to the full relation $M^k$).

We will use the observation above to modify our translation (in \cite{Ronnholm15}) from ESO[$k$] to INEX[$k$]. If we only consider \emph{sentences} of exclusion logic, we can quantify the needed complementary values, and then replace inclusion atoms in the translation with the corresponding exclusion atoms. The remaining problem is that in our translation we also needed a new connective called \emph{term value preserving disjunction} (\cite{Ronnholm15}) to avoid the loss of information on the values of certain variables when evaluating disjunctions. This operator can be defined by using both inclusion and exclusion atoms (\cite{Ronnholm15}), but it is undefinable in exclusion logic since it is not closed downwards.

In \cite{Ronnholm15} we have introduced new operators called \emph{inclusion and exclusion quantifiers} and defined them in inclusion-exclusion logic. Furthermore, we have shown that \emph{universal inclusion quantifier} $(\Ae\tuple x\inc\tuple t\,)$ could be defined also in exclusion logic. A natural reading for this quantifier is: ``for all the values of $\tuple x$ that are included in the values of $\tuple t$\,''. We will now consider the use of this quantifier in somewhat trivial looking form $(\Ae\tuple x\inc\tuple x)$. This operator turns out to be useful as it ``unifies'' the values of variables in a team. We will use it to define new operators called \emph{unifier}, \emph{unified existential quantifier} and \emph{unifying disjunction}.

Unifying disjunction will give us an alternative method to avoid the loss of information in the translation from ESO[$k$]. This completes our translation and proves the equivalence between EXC[$k$] and ESO[$k$] on the level of sentences. Hence we also obtain a strict arity hierarchy for exclusion logic since the arity hierarchy for ESO is known to be strict. We also get an interesting consequence that $k$-ary inclusion logic is \emph{strictly} weaker than EXC[$k$] on the level of sentences (for any $k\geq 1$).

Finally, we will examine the expressive power of inclusion logic with an alternative semantics, so-called \emph{strict semantics}. This semantical variant of inclusion logic has stronger expressive power by capturing the whole ESO, as shown by Galliani, Hannula and Kontinen \cite{Kontinen13a}, but lacks some nice semantical properties. We will use similar ideas, as in our translation from ESO[$k$] to EXC[$k$], to formulate a translation from ESO[$k$] to INC[$k$] with strict semantics. Consequently, for any arity fragment of inclusion logic, strict semantics is more expressive than the standard semantics.

The structure of this paper is as follows: After preliminaries in Section~\ref{sec: Preliminaries}, we present various new operators
for exclusion logic Section~\ref{sec: Operators}. In Section~\ref{sec: Expressive power} we prove our main result by forming a translation from ESO[$k$] to EXC[$k$] on the level of sentences. Finally, in Section~\ref{sec: Expressive power of INCs}, we present a translation from ESO[$k$] to INC[$k$] with strict semantics. This paper is an extended journal version of \cite{Ronnholm16a} with more detailed proofs and additional examples. Moreover, all the results in Section~\ref{sec: Expressive power of INCs} are previously unpublished work.


\section{Preliminaries}\label{sec: Preliminaries}

In this section we first present the team semantics for FO. Then we define inclusion and exclusion logics and review some of their known properties.


\subsection{Syntax and team semantics for $\FO$}

\emph{A vocabulary} $L$ is a set of \emph{relation symbols}~$R$, \emph{function symbols} $f$ and \emph{constant symbols} $c$. The set of \emph{$L$-terms}, $\Tset$, is defined in the standard way. The set of variables occurring in a tuple $\tuple t$ of $L$-terms is denoted by $\vr(\tuple t\,)$. 
%
%
\begin{definition}\label{def: FOset}
The set of $\FOset$-formulas is defined as follows:
\[
	\varphi::=\, t_1\!=\!t_2 \mid \neg t_1\!=\!t_2 \mid R\,\tuple t \mid \neg R\,\tuple t 
	\mid (\varphi\wedge\varphi) \mid (\varphi\vee\varphi) \mid \Ee x\,\varphi \mid \Ae x\,\varphi
\]
$\FOset$-formulas of the form $t_1\!=\!t_2$, $\neg t_1\!=\!t_2$, $R\,\tuple t$ and $\neg R\,\tuple t$ are called \emph{literals}.
\end{definition}
Let $\varphi\in\FOset$. We denote the set of \emph{subformulas} of $\varphi$ by $\subf(\varphi)$, the set of variables occurring in $\varphi$ by $\vr(\varphi)$ and the set of \emph{free variables} of $\varphi$ by $\fr(\varphi)$.


An \emph{$L$-model} $\mathcal{M}=(M,\mathcal{I})$, where the \emph{universe} $M$ is any nonempty set and the \emph{interpretation} $\mathcal{I}$ is a function whose domain is the vocabulary~$L$. The interpretation $\mathcal{I}$ maps constant symbols to elements in $M$, $k$-ary relation symbols to $k$-ary relations in $M$ and $k$-ary function symbols to functions $M^k\rightarrow M$. For all $k\in L$ we write $k^\mathcal{M}:=\mathcal{I}(k)$. An \emph{assignment} $s$ for $M$ is a function that is defined in some set of variables, $\dom(s)$, and ranges over $M$. A \emph{team} $X$ for $M$ is any set of assignments for $M$ with a common domain, denoted by $\dom(X)$.

Let $s$ be an assignment and $a$ be any element in $M$. The assignment $s[a/x]$ is defined in $ \dom(s)\cup\{x\}$, and it maps the variable $x$ to $a$ and all other variables as $s$. If $\tuple x:= x_1\dots x_k$ is a tuple of variables and $\tuple a := (a_1,\dots,a_k)\in M^k$, we write $s[\tuple a/\tuple x\,] := s[a_1/x_1,\dots,a_k/x_k]$. For a team $X$, a set $A\subseteq M^k$ and for a function $\mathcal{F}:X\rightarrow\mathcal{P}(M^k)\setminus\{\emptyset\}$ we use the following notations.
\[
	\begin{cases}
		X[A/\tuple x\,]  := \bigl\{s[\tuple a/\tuple x\,]\mid s\in X,\,\tuple a\in A\bigr\} \\
		X[\mathcal{F}/\tuple x\,]  := \bigl\{s[\tuple a/\tuple x\,]\mid s\in X,\,\tuple a\in\mathcal{F}(s)\bigr\}.
	\end{cases}
\]
Let $\mathcal{M}$ be an $L$-model, $s$ an assignment and $t$ an $L$-term s.t. $\vr(t)\subseteq\dom(s)$. The \emph{interpretation of $t$ with respect to $\mathcal{M}$ and $s$} is denoted simply by $s(t)$. For a team $X$ and a tuple $\tuple t:= t_1\dots t_k$ of $L$-terms we write
\[
	s(\tuple t\,) := \,(s(t_1),\dots,s(t_k)) \;\text{ and }\; X(\tuple t\,) := \,\{s(\tuple t\,)\mid s\in X\}.
\]

If $A\subseteq M^k$, we write $\overline A :=M^k\setminus A$. We are now ready to define \emph{team semantics} for first order logic.
\begin{definition}\label{def: Team semantics}
Let $\mathcal{M}$ be an $L$-model, $\varphi\in\FOset$ and $X$ a team such that $\fr(\varphi)\subseteq\dom(X)$. We define the \emph{truth of $\varphi$ in the model $\mathcal{M}$ and the team $X$}, $\mathcal{M}\true_X\varphi$, as follows.
\begin{itemize}
\item $\mathcal{M}\true_X t_1\!=\!t_2$ iff \,$s(t_1)=s(t_2)$ for all $s\in X$.
\item $\mathcal{M}\true_X \neg t_1\!=\!t_2$\, iff \,$s(t_1)\neq s(t_2)$ for all $s\in X$.
\item $\mathcal{M}\true_X R\,\tuple t$\, iff\, $X(\tuple t\,)\subseteq R^\mathcal{M}$.
\item $\mathcal{M}\true_X \neg R\,\tuple t$\, iff\, $X(\tuple t\,)\subseteq \overline{R^\mathcal{M}}$.
\item $\mathcal{M}\true_X \psi\wedge\theta$\, iff \,$\mathcal{M}\true_X \psi$ and $\mathcal{M}\true_X \theta$.
\item $\mathcal{M}\true_X \psi\vee\theta$\, iff \,there are $Y,Y'\!\subseteq\!X$ s.t. $Y\cup Y'=X$,
	$\mathcal{M}\true_Y \psi$ and $\mathcal{M}\true_{Y'} \theta$.
\item $\mathcal{M}\true_X\Ee x\,\psi$\, iff \,there exists $F:X\rightarrow\mathcal{P}(M)\setminus\{\emptyset\}$\,
	 s.t. $\mathcal{M}\true_{X[F/x]}\psi$.
\item $\mathcal{M}\true_X\Ae x\,\psi$\, iff \,$\mathcal{M}\true_{X[M/x]}\psi$.
\end{itemize}
\end{definition}
\begin{remark}
Above we have defined so-called \emph{lax-semantics} in which we may select several witnesses when existentially quantifying variables and we may allow the ``witnessing teams'' ($Y$ and $Y'$) for disjunction to overlap. These operators also have an alternative so-called \emph{strict semantics}:

\smallskip

\begin{itemize}
\item $\mathcal{M}\true_X \psi\vee\theta$\, iff \,there are $Y,Y'\!\subseteq\!X$ s.t. $Y\cup Y'=X$, \\ 
\text{}\hspace{3cm}$Y\cap Y'=\emptyset$,
$\mathcal{M}\true_Y \psi$ and $\mathcal{M}\true_{Y'} \theta$.

\smallskip

\item $\mathcal{M}\true_X\Ee x\,\psi$  iff there is $F:X\rightarrow M$ s.t. $\mathcal{M}\true_{X[F/x]}\psi$, \\ 
\text{}\hspace{3cm}where $X[F/x]$ is the team $\{x[F(s)/x]\mid s\in X\}$.
\end{itemize}

\smallskip

For FO these two semantic variants are equivalent. Galliani \cite{Galliani12b} has shown that they are also equivalent for exclusion logic but not for inclusion logic.
\end{remark}
\noindent
For tuples $\tuple t:= t_1\dots t_k$ and $\tuple{t'}:= t_1'\dots t_k'$ of $L$-terms we write
\begin{align*}
	\tuple t \!=\! \tuple{t'} \; &:= \; \bigwedge_{i\leq k} t_i\!=\!t_i'
	\qquad\text{and}\qquad
	\tuple t \!\neq\! \tuple{t'} \; := \; \bigvee_{i\leq k} \neg t_i\!=\!t_i'.
\end{align*}
It is easy to see that the following equivalences hold:
\begin{align*}
	\mathcal{M}\true_X\tuple t = \tuple{t'}\, &\;\text{ iff } s(\tuple t\,) = s(\tuple{t'}) \text{ for all } s\in X \\
	\mathcal{M}\true_X\tuple t \neq \tuple{t'}\, &\;\text{ iff } s(\tuple t\,) \neq s(\tuple{t'}) \text{ for all } s\in X. 
\end{align*}

For $\varphi\in\FOset$ and tuple $\tuple x:= x_1\dots x_k$ we write: $\Ee\tuple x\,\varphi := \Ee x_1\dots\Ee x_k\varphi$ and $\Ae\tuple x\,\varphi := \Ae x_1\dots\Ae x_k\varphi$. It is easy to show that
\begin{itemize}
\item $\mathcal{M}\true_X\Ee\tuple x\,\varphi$\, iff there exists $\mathcal{F}:X\rightarrow\mathcal{P}(M^k)\setminus\{\emptyset\}$ s.t. $\mathcal{M}\true_{X[\mathcal{F}/\tuple{x}\,]}\varphi$.
\smallskip
\item $\mathcal{M}\true_X\Ae\tuple x\,\varphi$\, iff $\mathcal{M}\true_{X[M^k/\tuple{x}\,]}\varphi$.
\end{itemize}
In strict semantics the first condition turns into the form: \;$\mathcal{M}\true_X\Ee\tuple x\,\varphi$\, iff there exists $\mathcal{F}:X\rightarrow M^k$ s.t. $\mathcal{M}\true_{X[\mathcal{F}/\tuple{x}\,]}\varphi$, where $X[\mathcal{F}/\tuple x\,]  := \{s[\mathcal{F}(s)/\tuple x\,]\mid s\in X\}$.
First order logic with team semantics has so-called \emph{flatness}-property:
\begin{proposition}[\cite{Vaananen07}, Flatness]\label{the: Flatness}
Let $X$ be a team and $\varphi\in\FOset$. The following equivalence holds:
\;$\mathcal{M}\true_X\varphi\, \,\text{ iff }\; \mathcal{M}\true_{\{s\}} \varphi \,\text{ for all } s\in X$.
\end{proposition}
We use notations $\true_s^\text{T}$ and $\true^\text{T}$ for truth in a model with standard Tarski semantics. Team semantics can be seen just as a generalization of Tarski semantics as shown by the following proposition.
\begin{proposition}[\cite{Vaananen07}]\label{the: Tarski}
The following equivalences hold:
\begin{align*}
	\mathcal{M}\true_s^\text{\emph{T}}\varphi\, &\text{ iff }\, \mathcal{M}\true_{\{s\}}\varphi
	\quad\text{ for all $\FOset$-formulas } \varphi \text{ and assignments } s. \\[-0,05cm]
	\mathcal{M}\true^\text{\emph{T}}\varphi\, &\text{ iff }\, \mathcal{M}\true_{\{\emptyset\}}\varphi
	\quad\text{ for all $\FOset$-sentences } \varphi.
\end{align*}
\end{proposition}
Note that, by flatness, $\mathcal{M}\true_X\varphi$ if and only if $\mathcal{M}\true_s^\text{T}\varphi$ for all $s\in X$. By Proposition~\ref{the: Tarski} it is natural to write $\mathcal{M}\true\varphi$, when we mean that $\mathcal{M}\true_{\{\emptyset\}}\varphi$. Note that $\mathcal{M}\true_\emptyset\varphi$ holds trivially for all $\FOset$-formulas $\varphi$ by Definition \ref{def: Team semantics}. In general we say that any logic $\mathcal{L}$ with team semantics has the \emph{empty team property} if $\mathcal{M}\true_\emptyset\varphi$ holds for all $\mathcal{L}$-formulas~$\varphi$. We define three more important properties for any logic $\mathcal{L}$ with team semantics.
\begin{definition}
Let $\mathcal{L}$ be any logic with team semantics. We say that
\begin{itemize}
\item $\mathcal{L}$ is \emph{local}, if the truth of formulas is determined only by the values of their free variables in a team, i.e. we have: \;$\mathcal{M}\true_X\varphi\, \text{ iff } \mathcal{M}\true_{X\upharpoonright\fr(\varphi)} \varphi$.
\smallskip
\item $\mathcal{L}$ is \emph{closed downwards} if: \;$\mathcal{M}\true_X\varphi \text{ and } Y\subseteq X \Rightarrow \mathcal{M}\true_Y\varphi$.
\item $\mathcal{L}$ is \emph{closed under unions} if: \;$\mathcal{M}\true_{X_i}\varphi \text{ for every } i\in I \Rightarrow \mathcal{M}\true_{\cup_{i\in I}X_i}\varphi$.
\end{itemize}
\smallskip
By flatness it is easy to see that FO is local and closed both downwards and under unions.
\end{definition}


\subsection{Inclusion and exclusion logics}

Inclusion logic ($\INC$) and exclusion logic ($\EXC$) are obtained by adding inclusion and exclusion atoms, respectively, to FO with team semantics.
\begin{definition}\label{def: Inclusion and exclusion logics}
If $\tuple t_1,\tuple t_2$ are $k$-tuples of $L$-terms, $\tuple t_1\inc\tuple t_2$ is a \emph{$k$-ary inclusion atom}. $\INCset$-formulas are formed like $\FOset$-formulas by allowing the use of (non-negated) inclusion atoms like literals.
Let $\mathcal{M}$ be a model and $X$ a team s.t. $\vr(\tuple t_1\tuple t_2)\subseteq\dom(X)$. We define the truth of $\tuple t_1\subseteq\tuple t_2$ in $\mathcal{M}$ and $X$ as:
\[
	\mathcal{M}\true_X \tuple t_1\subseteq\tuple t_2\;\, \text{ iff \;for all } s\in X \text{ there exists } s'\in X 
	\text{ s.t. } s(\tuple t_1)=s'(\tuple t_2).
\]
Equivalently we have $\mathcal{M}\true_X \tuple t_1\subseteq\tuple t_2$ \,iff\, $X(\tuple t_1)\subseteq X(\tuple t_2)$.

\medskip

If $\tuple t_1,\tuple t_2$ are $k$-tuples of $L$-terms, $\tuple t_1\,|\,\tuple t_2$ is a \emph{$k$-ary exclusion atom}. $\EXCset$-formulas are formed as $\FOset$-formulas, but (non-negated) exclusion atoms may be used as literals are used in $\FO$.
Let $\mathcal{M}$ be a model and $X$ a team for which we have $\vr(\tuple t_1\tuple t_2)\subseteq\dom(X)$. We define the truth of $\tuple t_1\mid\tuple t_2$ in $\mathcal{M}$ and $X$ as:
\[
	\mathcal{M}\true_X \tuple t_1\mid\tuple t_2\;\, \text{ iff \;for all } s,s'\in X: \, s(\tuple t_1)\neq s'(\tuple t_2).
\]
Equivalently we have $\mathcal{M}\true_X \tuple t_1\mid\tuple t_2$ \,iff\, $X(\tuple t_1)\cap X(\tuple t_2) = \emptyset$ \;$(\,\text{iff } X(\tuple t_1)\subseteq\overline{X(\tuple t_2)}\,)$.
\end{definition}
\smallskip
Inclusion-exclusion logic (INEX) is defined simply by allowing the use of both inclusion and exclusion atoms. If $\varphi\in\EXCset$ contains at most $k$-ary exclusion atoms, we say that $\varphi$ is a formula of \emph{$k$-ary exclusion logic}, $\EXC[k]$. Moreover, \emph{$k$-ary inclusion logic} (INC[$k$]) and \emph{$k$-ary inclusion-exclusion logic} (INEX[$k$]) are defined analogously.

The following properties have all been shown by Galliani~\cite{Galliani12b}: $\EXC$, $\INC$ and $\INEX$ are all local and satisfy the empty team property. $\EXC$ is also closed downwards, unlike $\INC$ which is closed under unions.
If we use strict semantics for INC, the resulting logic is not local. This is one of the reasons why the (lax)-semantics given in Definition~\ref{def: Team semantics} is usually considered to be more natural.

We denote inclusion logic with strict semantics by INC$^s$ and its $k$-ary fragment by INC$^s[k]$. 
%
%
Galliani, Hannula and Kontinen \cite{Kontinen13a} have shown INC$^s$ is equivalent with ESO. Thus, on the level of sentences, INC$^s$ is equivalent with exclusion logic and stronger than (the standard) inclusion logic. We will study the properties of INC$^s[k]$ in Section~\ref{sec: Expressive power of INCs}.


\section{Useful operators for exclusion logic}\label{sec: Operators}

In this section we will define several operators for EXC[$k$]. We first review how \emph{$k$-ary dependence atoms} and \emph{intuitionistic disjunction} can be expressed in EXC[$k$]. Then we show how the values of certain tuples of terms can be \emph{unified} by using \emph{universal inclusion quantifier} that can be defined for EXC. With this technique we can define other useful operators for this framework.


\subsection{Dependence atoms and intuitionistic disjunction}

Let us review the semantics for \emph{dependence atoms} of \emph{dependence logic} (\cite{Vaananen07}). Let $t_1\dots t_k$ be $L$-terms. The $k$-ary dependence atom $\dep(t_1\dots t_{k-1},t_k)$ has the following truth condition: $\mathcal{M}\true_X\dep(t_1\dots t_{k-1},t_k)$ if and only if we have:
\[
	\text{for all } s,s'\in X \text{ for which } s(t_1\dots t_{k-1}) =  s'(t_1\dots t_{k-1}) \text{ also } s(t_k) = s'(t_k),
\]
for all $L$-models $\mathcal{M}$ and teams $X$ for which $\vr(t_1\dots t_k)\subseteq\dom(X)$. This truth condition can be read as follows: ``the value of $t_k$ is (functionally) dependent on the values of $t_1,\dots,t_{k-1}$''.
By using Galliani's translation between dependence logic and exclusion logic, we can express $k$-ary dependence atoms in EXC[$k$]:
\begin{proposition}[\cite{Galliani12b}]\label{the: Expressing dependence atom}
Let $\tuple t=t_1\dots t_k$ be a tuple of $L$-terms. The $k$-ary dependence atom $\dep(t_1\dots t_{k-1},t_k)$ is equivalent with the $\EXCset[k]$-formula $\varphi$:
\[
	\varphi :=\; \Ae x\,(x = t_k\,\vee\;t_1\dots t_{k-1}x \mid \tuple t\,), \; \text{ where $x$ is a fresh variable.}
\]
\end{proposition}

Hence, in particular, we can express \emph{constancy atom}\footnote{$\dep(t)$ is true in a nonempty team $X$ \,iff\, $t$ has a constant value in $X$, i.e. $\abs{X(t)}=1$.} $\dep(t)$ in EXC[$k$] for any $k\geq 1$. 
The semantics of \emph{intuitionistic disjunction} $\sqcup$ is obtained by lifting the Tarski semantics of classical disjunction from single assignments to teams. That is, $\mathcal{M}\true_X \varphi\sqcup\psi \;\text{ iff }\; \mathcal{M}\true_X\varphi \text{ or } \mathcal{M}\true_X\psi$.
Galliani \cite{Galliani12a} has shown that this operator can be expressed by using constancy atoms. Hence we can can define it as an abbreviation in $\EXC[k]$ for any $k\geq 1$.


\subsection{Universal inclusion quantifier and unifier}\label{ssec: Uniformization}

In \cite{Ronnholm15} we have considered inclusion and exclusion dependencies from a new perspective by introducing inclusion and exclusion quantifiers. These quantifiers range over values of certain terms (or their complements) in the team instead of the whole universe $M$. We review here the semantics for \emph{universal} inclusion and exclusion quantifiers $(\Ae\tuple x\inc\tuple t\,)$ and $(\Ae\tuple x\exc\tuple t\,)$. Let $\tuple x$ be a $k$-tuple of variables, $\tuple t$ a $k$-tuple of $L$-terms and $\varphi\in\INEXset$. Now $(\Ae\tuple x\inc\tuple t\,)$ and $(\Ae\tuple x\exc\tuple t\,)$ have the following truth conditions: 
\begin{align*}
	\mathcal{M}\true_X (\Ae\tuple x\inc\tuple t\,)\,\varphi
	&\;\text{ iff }\, \mathcal{M}\true_{X[A/\tuple x\,]}\varphi, \text{ where } A = X(\tuple{t}\,). \\
	\mathcal{M}\true_X (\Ae\tuple x\exc\tuple{t}\,)\,\varphi 
	&\;\text{ iff }\, \mathcal{M}\true_{X[A/\tuple x\,]}\varphi, \text{ where }A = \overline{X(\tuple{t}\,)}.	
\end{align*}

The quantifier $(\Ae\tuple x\inc\tuple t\,)$ has a natural reading: ``for all tuples $\tuple x$ that are included in the values of $\tuple t$\,''. And likewise $(\Ae\tuple x\exc\tuple t\,)$ can be read as: ``for all tuples $\tuple x$ that are excluded of the values of $\tuple t$\,''.
These quantifiers can be defined in INEX by using the following idea: we first universally quantify $\vec x$ and then use inclusion and and exclusion atoms along with disjunction to force the team to be split into subteams $X[X(\vec t)/\vec x\,]$ and $X[\overline{X(\vec t)}/\vec x\,]$; then we just state that $\varphi$ holds in the corresponding subteam (see \cite{Ronnholm15} for more details).

In order to define these quantifiers as abbreviations in $\INEX$ we needed to use \emph{both} $k$-ary inclusion and exclusion atoms (see \cite{Ronnholm15} for details). However, we can alternatively define a quantifier $(\Ae\tuple x\ince\tuple t)$ as an abbreviation by using only $k$-ary  exclusion atoms (\cite{Ronnholm15}). This quantifier has the same truth condition as $(\Ae\tuple x\inc\tuple t\,)$ above, when $\varphi$ is a formula of \emph{exclusion logic}.
Hence the universal inclusion quantifier for $k$-tuples of variables can be defined for both $\INEX[k]$ and $\EXC[k]$, although these definitions have to be given differently. From now on we will always use the plain notation $(\Ae\tuple x\inc\tuple t\,)$ and assume it be defined in the right way depending on whether we use it with $\INEX$ or $\EXC$.  

When defining quantifier $(\Ae\tuple x\inc\tuple t\,),$ we allowed the variables in the tuple $\tuple x$ to occur in $\vr(\tuple t\,)$. In particular, we accept the quantifiers of the form $(\Ae \tuple x\inc\tuple x)$. Quantifiers of this form may seem trivial, but they turn out to be rather useful operators. Let us analyze their truth condition:
\[
	\mathcal{M}\true_X (\Ae\tuple x\inc\tuple x)\,\varphi \;\text{ iff }\, \mathcal{M}\true_{X'}\varphi,
	\text{ where } X'=X[X(\tuple x)/\tuple x\,].
\] 
Note that the team $X'$ is not necessarily the same team as $X$, although we have $\dom(X')=\dom(X)$ and even $X'(\tuple x)=X(\tuple x)$. Consider the following example.
\begin{example}
Let $X\!=\!\{s_1,s_2\}$ where $s_1(v_1)\!=\!a$, $s_2(v_1)=b$ and $a\neq b$. Now
\begin{align*}
	X[X(v_1)/v_1] = X[\{a,b\}/v_1] &= \{s_1[a/v_1],s_1[b/v_1],s_2[a/v_1],s_2[b/v_1]\} \\
	&= \{s_1,s_2,s_1[b/v_1],s_2[a/v_1]\}.
\end{align*}
Thus, supposing that there is at least one variable $x\in\dom(X)$ for which $x\neq v_1$ and $s_1(x)\neq s_2(x)$, we have $X[X(v_1)/v_1]\neq X$.
\end{example}
We say that the quantifier $(\Ae \tuple x\inc\tuple x)$ \emph{unifies} the values of the tuple $\tuple x$ in a team. After executing this operation for a team $X$, then each assignment $s\in X\upharpoonright(\dom(X)\setminus\vr(\tuple x))$ ``carries'' the information on the whole relation $X(\tuple x)$. 
This also makes the values of the tuple $\tuple x$ \emph{independent} of all the other variables in $\dom(X)$. We can formulate this latter statement  in \emph{independence logic} (\cite{Gradel12}) as follows, when $\vr(\tuple x)\cap\vr(\tuple v)=\emptyset$.
\[
	\mathcal{M}\true_X(\Ae \tuple x\inc\tuple x)\,\tuple x\bot\tuple v 
	\quad \text{holds in any team $X$ for which $\vr(\tuple v)\subseteq\dom(X)$}.
\]
We introduce the following operator as an abbreviation.
\begin{definition}\label{def: Unifier}
Let $\tuple x_1,\dots,\tuple x_n$ be tuples of variables and $\varphi\in\EXCset$. \emph{The unifier of the values of $\tuple x_1,\dots,\tuple x_n$}, denoted by $\uniform(\tuple x_1,\dots,\tuple x_n)$, is defined as:
\[
	\uniform(\tuple x_1,\dots,\tuple x_n)\,\varphi :=
	\;(\Ae\tuple x_1\inc\tuple x_1)\dots(\Ae\tuple x_n\inc\tuple x_n)\,\varphi.
\]
\end{definition}
Note that tuples $\tuple x_1,\dots,\tuple x_n$  above do not necessarily need to be of the same length. Moreover, they do not have to be disjoint, i.e the same variable may occur in more than one tuple.
Also note that if the longest of the tuples $\tuple x_i$ is a $k$-tuple, then this operator can be defined in $\EXC[k]$ (and in $\INEX[k]$).

\begin{example}\label{ex: Unifier}
We have $\uniform(\tuple x_1,\dots,\tuple x_n)\varphi\equiv\uniform(\tuple x_1)\dots\uniform(\tuple x_n)\,\varphi$ by the definition of the unifier. But one should note that usually
\[
	\uniform(\tuple x_1\!\dots\tuple x_n)\,\varphi\not\equiv\uniform(\tuple x_1,\,\dots\,,\tuple x_n)\,\varphi. 
\]
To see this, consider $X$ s.t. $v_1,v_2\in\dom(X)$ and let $X_1:=X[X(v_1v_2)/v_1v_2]$ and $X_2:=X[X(v_1)/v_1,X(v_2)/v_2]$. Now we have $X_1(v_1v_2)=X(v_1v_2)$ but $X_2(v_1v_2)=X(v_1)\times X(v_2)$. It is easy to see that $X_1$ and $X_2$ are identical only if $X(v_1v_2)=X(v_1)\times X(v_2)$.

We also note that the ordering of variables \emph{within the tuples} does not effect the truth condition of the unifier. Hence for example $\uniform(x_1x_2)\varphi\equiv\uniform(x_2x_1)\varphi$ for any formula $\varphi$. Also clearly the repetitions of variables within the tuples do not matter, and thus for example $\uniform(x_1x_1)\varphi\equiv\uniform(x_1)\varphi$ for any $\varphi$.
\end{example}

The truth condition for the unifier is given by the following proposition whose truth is obvious.

\begin{proposition}
Let $\tuple x_1,\dots,\tuple x_n$ be tuples of variables and $\varphi\in\EXCset$. Now
\[
	\mathcal{M}\true_X \uniform(\tuple x_1,\dots,\tuple x_n)\,\varphi \; \text{ iff } \;
	\mathcal{M}\true_{X[X(\tuple x_1)/\tuple x_1,\dots,X(\tuple x_n)/\tuple x_n]}\varphi.
\]
\end{proposition}

When looking at Definition~\ref{def: Unifier}, it seems that if the tuples $\tuple x_1,\dots,\tuple x_n$ are not disjoint, then their ordering affects the truth condition of $\uniform(\tuple x_1,\dots,\tuple x_n)$. 
%
%
However, we can show that it is actually irrelevant in which order we unify the tuples. 
We first prove the following proposition which shows what happens when we unify two tuples which have some shared variables. The result shows that we obtain then the same result as when unifying separately the part that is overlapping and the disjoint parts.

\begin{proposition}\label{the: Unifying overlapping tuples}
Let $x_1,\dots,x_k$ be distinct variables and $1 < m \leq n < k$. Then the following equivalence holds.
\begin{align*}
	&\mathcal{M}\true_X \uniform(x_1\dots x_n)\uniform(x_m\dots x_k)\,\varphi \\
	&\qquad\text{ iff } \mathcal{M}\true_X \uniform(x_1\dots x_{m-1})\uniform(x_m\dots x_n)\uniform(x_{n+1}\dots x_k)\,\varphi.
\end{align*}
\end{proposition}

\begin{proof}
We define the following teams:
\begin{align*}
	X_1 &:= X[X(x_1\dots x_n)/x_1\dots x_n] \\
	X_2 &:= X_1[X_1(x_m\dots x_k)/x_m\dots x_k] \\
	X_3 &:= X[X(x_1\dots x_{m-1})/x_1\dots x_{m-1}, X(x_m\dots x_n)/x_m\dots x_n, \\[-0,1cm]
		&\hspace{6cm}X(x_{n+1}\dots x_k)/x_{n+1}\dots x_k].
\end{align*}
By the semantics of universal inclusion quantifier, it is sufficient to show that $X_2=X_3$.

For the sake of showing that $X_3\subseteq X_2$, let $s\in X_3$. Hence there is $r\in X$ and  $a_1\dots a_{m-1}\in X(x_1\dots x_{m-1})$, $a_m\dots a_n\in X(x_m\dots x_n)$ and  $a_{n+1}\dots a_k\in X(x_{n+1}\dots x_k)$ such that 
\[
	s=r[a_1\dots a_{m-1}/x_1\dots x_{m-1},\,a_m\dots a_n/x_m\dots x_n,\,a_{n+1}\dots a_k/x_{n+1}\dots x_k].
\]
Moreover, there are $r_1,r_2,r_3\in X$ such that $r_1(x_1,\dots x_{m-1})=a_1\dots a_{m-1}$, $r_2(x_m,\dots x_n)=a_m\dots a_n$ and $r_3(x_{n+1},\dots x_k)=a_{n+1}\dots a_k$.

Let now $r':=r[r_1(x_1\dots x_n)/x_1\dots x_n]$ and $r'':=r_3[r_2(x_1\dots x_n)/x_1\dots x_n]$, whence $r',r''\in X_1$. It is quite easy to see that $s=r'[r''(x_m\dots x_k)/x_m\dots x_k]$ and therefore $s\in X_2$.

For the sake of showing that $X_2\subseteq X_3$, let $s\in X_2$. Now there is $r_1\in X_1$ and $a_m\dots a_k\in X_1(x_m\dots x_k)$ such that $s=r_1[a_m\dots a_k/x_m\dots x_k]$. Moreover, there is $r_2\in X$ and $a_1'\dots a_n'\in X(x_1\dots x_n)$ such that $r_1=r_2[a_1'\dots a_n'/x_1\dots x_n]$. Since $a_m\dots a_k\in X_1(x_m\dots x_k)$, there is $r_1'\in X_1$ s.t. $r_1'(x_m\dots x_k)=a_m\dots a_k$. Furthermore there is $r_2'\in X$ and a tuple $a_1''\dots a_n''\in X(x_1\dots x_n)$ such that $r_1'=r_2'[a_1''\dots a_n''/x_1\dots x_n]$. Let $a_i''':=r_2'(x_i)$ for each $i$ s.t. $n<i\leq k$.
Now
\[
	s=r_2[a_1'\dots a_{m-1}'/x_1\dots x_{m-1},\,a_m''\dots a_n''/x_m\dots x_n,\,a_{n+1}'''\dots a_k'''/x_{n+1}\dots x_k].
\]
Because we have $a_1'\dots a_{m-1}'\in X(x_1\dots x_{m-1})$, $a_m''\dots a_n''\in X(x_m\dots x_n)$ and $a_{n+1}'''\dots a_k'''\in X(x_{n+1}\dots x_k)$, it holds that $s\in X_3$.
\end{proof}

Consider now some arbitrary tuples $\vec x_1$ and $\vec x_2$ of variables. Recalling the observations in Example~\ref{ex: Unifier}, we can first rearrange these tuples in such a way that they match the assumptions of Proposition~\ref{the: Unifying overlapping tuples} and then unify the overlapping and disjoint parts separately. 
When unifying several tuples that are not disjoint, we can then show by a straightforward induction that one always obtains the same result by separately unifying some disjoint tuples. 
It then follows that the ordering of tuples $\tuple x_1,\dots,\tuple x_n$ in $\uniform(\tuple x_1,\dots,\tuple x_n)$ indeed does not affect its truth condition.
For example we have
\[
	\uniform(v_1 v_2,\,v_2 v_3)\,\varphi 
	\,\equiv\, \uniform(v_1,v_2,v_3)\,\varphi 
	\,\equiv\, \uniform(v_3 v_2,\,v_2 v_1)\,\varphi
	\,\equiv\, \uniform(v_2 v_3,\,v_1 v_2)\,\varphi.
\]

For the main results of this paper we only use the unifier for disjoint tuples and in this case the result of Proposition~\ref{the: Unifying overlapping tuples} is not needed. However, we think that this is an interesting property which could be useful when using unifier in some other context.


\subsection{New operators that can be defined with unifier}

Unifier can be used in combination with other logical operators to form new useful tools for the framework of team semantics. We will introduce here two such operators. The definitions for the following operators are given more generally for INEX, but they can be defined in the same way for EXC as well.
\begin{definition}
Let $\tuple x$ be a $k$-tuple of variables and $\varphi\in\INEXset$. \emph{Unified existential quantifier} $\Eu$ is defined as:
\[
	\Eu \tuple x\,\varphi\, :=\;\Ee\tuple x\,\uniform(\tuple x)\,\varphi.
\]
\end{definition}
\begin{proposition}\label{the: Unified existential quantifier}
Let $\tuple x$ be a $k$-tuple and $\varphi\in\INEXset$. Now
\[
	\mathcal{M}\true_{X} \Eu\tuple x\,\varphi	
	\,\text{ iff\, there exists a nonempty set } A\subseteq M^k \text{ s.t. }\mathcal{M}\true_{X[A/\tuple x\,]}\varphi.
\]
\end{proposition}
\begin{proof}
If $X$ were the empty team, then the claim would hold trivially. Thus we may assume that $X\neq\emptyset$.

Suppose first that we have $\mathcal{M}\true_X \Eu\tuple x\,\varphi$, i.e. $\mathcal{M}\true_X \Ee\tuple x\uniform(\tuple x)\,\varphi$. Therefore there exists a function $\mathcal{F}:X\rightarrow\mathcal{P}(M^k)\setminus\{\emptyset\}$ s.t. $\mathcal{M}\true_{X'}\uniform(\tuple x)\,\varphi$, where $X'=X[\mathcal{F}/\tuple x\,]$. Then $\mathcal{M}\true_{X'[X'(\tuple x)/\tuple x\,]}\varphi$. Since $X[X'(\tuple x)/\tuple x\,]=X'[X'(\tuple x)/\tuple x]$ and $X'(\tuple x)\neq\emptyset$, we can choose $A:=X'(\tuple x)$.

Suppose then that there exists nonempty $A\subseteq M^k$ s.t. $\mathcal{M}\true_{X[A/\tuple x\,]}\varphi$. We define the function
\[
	\mathcal{F}:X\rightarrow\mathcal{P}(M^k)\setminus\{\emptyset\}, \quad s\mapsto A \quad \text{for all } s\in X.
\]
Let $X':=X[\mathcal{F}/\tuple x\,]$, whence $X'(\tuple x)\!=\!A$. Now $X'[X'(\tuple x)/\tuple x\,]\!=\!X'[A/\tuple x\,]\!=\!X[A/\tuple x\,]$. Hence $\mathcal{M}\true_{X'[X'(\tuple x)/\tuple x\,]}\varphi$, and thus $\mathcal{M}\true_{X'}\uniform(\tuple x)\,\varphi$. Therefore $\mathcal{M}\true_X \Ee\tuple x\uniform(\tuple x)\,\varphi$, i.e. $\mathcal{M}\true_X \Eu\tuple x\,\varphi$.
\end{proof}
If we use this quantifier in EXC (or in any other downwards closed logic), the following equivalence holds:
\[
	\mathcal{M}\true_X \Eu\tuple x\,\varphi\;
	\text{ iff\, there exists } \tuple a\in M^k \text{ s.t. }\mathcal{M}\true_{X[\{\tuple a\}/\tuple x\,]}\varphi.
\]
For single variables this truth condition is equivalent with the semantics of the quantifier $\Ee^1$ that was introduced in~\cite{Kontinen09}. Note that in dependence logic this quantifier can be defined simply as $\Ee^1 x \,\varphi\,:=\, \Ee x(\dep(x)\wedge\varphi)$. 

The next operator will play a very important role in our translation from ESO[$k$] to EXC[$k$] in the next section.

\begin{definition}\label{def: Unifying disjunction}
Let $\varphi,\psi\in\INEXset$ and let $\tuple x_1,\dots,\tuple x_n$ be $k$-tuples of disjoint variables. \emph{Unifying disjunction for tuples $\tuple x_1,\dots,\tuple x_n$} is defined as:
\begin{align*}
	\varphi\!\!\underset{\scriptscriptstyle\tuple x_1,\dots,\tuple x_n}{\;\vee^\text{\bf U}}\!\!\psi \,:=\; 
	&\Ee y_1\Ee y_2\uniform(\tuple x_1,\dots,\tuple x_n)
	\bigl((y_1\!=\!y_2\wedge\varphi)\,\vee\,(y_1\!\neq\!y_2\wedge\psi)\bigr), \\[-0,2cm]
	&\hspace{4cm}\text{ where $y_1,y_2$ are fresh variables}.
\end{align*}
\end{definition}
\begin{proposition}\label{the: Unifying disjunction}
Let $\varphi,\psi\in\INEXset$ and let $\tuple x_1,\dots,\tuple x_n$ be $k$-tuples of variables. Now for all $L$-models $\mathcal{M}$ with at least two elements we have
\begin{align*}
	&\mathcal{M}\true_X\,\varphi\!\!\underset{\scriptscriptstyle\tuple x_1,\dots,\tuple x_n}{\;\vee^\text{\bf U}}\!\!\psi
	\text{ iff }\text{there exist } Y,Y'\subseteq X \text{ s.t. } Y\cup Y' = X, \\[-0,05cm]
	&\qquad\qquad\mathcal{M}\true_{Y[X(\tuple x_1)/\tuple x_1,\dots,X(\tuple x_n)/\tuple x_n]}\varphi
	\;\text{ and }\; \mathcal{M}\true_{Y'[X(\tuple x_1)/\tuple x_1,\dots,X(\tuple x_n)/\tuple x_n]}\psi.
\end{align*}
\end{proposition}
The intuitive idea about the proof of Proposition~\ref{the: Unifying disjunction} is that before splitting the team, we must ``announce'' beforehand for each assignment if it will be placed on the left hand side or on the right hand side (or on both). This is done by giving the same or different values for the variables $y_1$ and $y_2$. Because the unification is done \emph{after} this announcement, but before the actual splitting of the team, all the values will be unified correctly on both sides.

\begin{proof} (Proposition~\ref{the: Unifying disjunction})
Because INEX is local, we may assume that $y_1,y_2\notin\dom(X)$.
Suppose first that $\mathcal{M}\true_X\,\varphi\!\!\underset{\scriptscriptstyle\tuple x_1,\dots,\tuple x_n}{\;\vee^\text{\bf U}}\!\!\psi$, i.e. 
\[
	\mathcal{M}\true_{X}\Ee y_1\Ee y_2\uniform(\tuple x_1,\dots,\tuple x_n)
	\bigl((y_1\!=\!y_2\wedge\psi)\,\vee\,(y_1\!\neq\!y_2\wedge\theta)\bigr). 
\]
Thus there exist $F_1:X\rightarrow\mathcal{P}(M)\setminus\{\emptyset\}$ and $F_2:X[F_1/y_1]\rightarrow\mathcal{P}(M)\setminus\{\emptyset\}$ s.t.
\[
	\mathcal{M}\true_{X_1}\uniform(\tuple x_1,\dots,\tuple x_n)
	\bigl((y_1\!=\!y_2\wedge\varphi)\,\vee\,(y_1\!\neq\!y_2\wedge\psi)\bigr),
\]	
where $X_1=X[F_1/y_1,\,F_2/y_2]$.
Therefore $\mathcal{M}\true_{X_2}(y_1\!=\!y_2\wedge\varphi)\,\vee\,(y_1\!\neq\!y_2\wedge\psi)$, where $X_2=X_1[X_1(\tuple x_1)/\tuple x_1,\dots,X_1(\tuple x_n)/\tuple x_n\,]$. Thus there exist $Z,Z'\subseteq X_2$ s.t. $Z\cup Z'=X_2$,  $\mathcal{M}\true_{Z}y_1\!=\!y_2\wedge\varphi$ and $\mathcal{M}\true_{Z'}y_1\!\neq\!y_2\wedge\psi$. Let 
\[
	\begin{cases}
		Y:=\,\{s\in X\mid \text{There exists $a\in M$ s.t. $s[a/y_1,a/y_2]\in X_1$}\} \\
		Y':=\,\{s\in X\mid \text{There exist $a,b\in M$ s.t. $a\!\neq\!b$ and $s[a/y_1,b/y_2]\in X_1$}\}.
	\end{cases}
\]
It is easy to see that $Y\cup Y'=X$. Also note that since $X(\tuple x_i)=X_1(\tuple x_i)$ for each $i\leq n$, it holds that $X_2=X_1[X(\tuple x_1)/\tuple x_1,\dots,X(\tuple x_n)/x_n\,]$. We will show that $Y[X(\tuple x_1)/\tuple x_1,\dots,X(\tuple x_n)/\tuple x_n]=Z\upharpoonright\dom(X)$.

Let $r\in Y[X(\tuple x_1)/\tuple x_1,\dots,X(\tuple x_n)/\tuple x_n]$. Now there exists $s\in Y$ and tuples $\tuple a_1\in X(\tuple x_1),\dots,\tuple a_n\in X(\tuple x_n)$ s.t. $r=s[\tuple a_1/\tuple x_1,\dots,\tuple a_n/\tuple x_n\,]$. Since $s\in Y$, there exists  $a\in M$ s.t. $q:=s[a/y_1,a/y_2]\in X_1$. Let $q':=q[\tuple a_1/\tuple x_1,\dots,\tuple a_n/\tuple x_n\,]$, whence $q'\in X_2$. Since $q'(y_1)=a=q'(y_2)$ and $\mathcal{M}\true_{Z'}y_1\!\neq\!y_2$ we have $q'\notin Z'$, and thus it must be that $q'\in Z$. But since now $r=q'\upharpoonright\dom(X)\,\in\,Z\upharpoonright\dom(X)$, we have shown that $Y[X(\tuple x_1)/\tuple x_1,\dots,X(\tuple x_n)/\tuple x_n]\subseteq Z\upharpoonright\dom(X)$.

Let then $r^*\in Z\upharpoonright\dom(X)$. Now there exists $r\in Z$ s.t. $r^*=r\upharpoonright\dom(X)$. Because $\mathcal{M}\true_Z y_1\!=\!y_2$ it must be that $r(y_1)=r(y_2)$. Since $r\in Z\subseteq X_2$ there exists $q\in X_1$ and $\tuple a_1\in X(\tuple x_1),\dots,\tuple a_n\in X(\tuple x_n)$ s.t. $r=q[\tuple a_1/\tuple x_1,\dots,\tuple a_n/\tuple x_n\,]$. Let $s:=q\upharpoonright\dom(X)$. Since $q(y_1)=r(y_1)=r(y_2)=q(y_2)$ and $s\in X$, by the definition of $Y$ we have $s\in Y$. Let $s':=s[\tuple a_1/\tuple x_1,\dots,\tuple a_n/\tuple x_n\,]$, whence $s'\in Y[X(\tuple x_1)/\tuple x_1,\dots,X(\tuple x_n)/\tuple x_n]$. But now it must also be that $s' = r^*$ and thus $Z\upharpoonright\dom(X)\subseteq Y[X(\tuple x_1)/\tuple x_1,\dots,X(\tuple x_n)/\tuple x_n]$.

We have shown that $Y[X(\tuple x_1)/\tuple x_1,\dots,X(\tuple x_n)/\tuple x_n]=Z\upharpoonright\dom(X)$. Since $\mathcal{M}\true_Z\varphi$, by locality $\mathcal{M}\true_{Z\upharpoonright\dom(X)}\varphi$ and thus $\mathcal{M}\true_{Y[X(\tuple x_1)/\tuple x_1,\dots,X(\tuple x_n)/\tuple x_n]}\varphi$. With a similar argumentation $Y'[X(\tuple x_1)/\tuple x_1,\dots,X(\tuple x_n)/\tuple x_n]=Z'\upharpoonright\dom(X)$ and consequently $\mathcal{M}\true_{Y'[X(\tuple x_1)/\tuple x_1,\dots,X(\tuple x_n)/\tuple x_n]}\psi$.

\medskip

Suppose then that there are subteams $Y,Y'\subseteq X$ such that $Y\cup Y' = X$, $\mathcal{M}\true_{Y[X(\tuple x_1)/\tuple x_1,\dots,X(\tuple x_n)/\tuple x_n]}\varphi$ and $\mathcal{M}\true_{Y'[X(\tuple x_1)/\tuple x_1,\dots,X(\tuple x_n)/\tuple x_n]}\psi$. Since $\abs{M}\geq 2$, there exist $a,b\in M$ s.t. $a\neq b$. We define the following functions:
\begin{align*}
	&F_1:X\rightarrow\mathcal{P}(M)\setminus\{\emptyset\}, \quad
	\begin{cases}
		s\mapsto\{a\} \;\quad\text{if } s\in Y\setminus Y' \\
		s\mapsto\{b\} \;\,\quad\text{if } s\in Y'\setminus Y\\
		s\mapsto\{a,b\} \;\;\text{if } s\in Y\cap Y'
	\end{cases} \\[0,2cm]
	&F_2:X[F_1/y_1]\rightarrow\mathcal{P}(M)\setminus\{\emptyset\}, \quad s\mapsto\{a\}.
\end{align*}
We define teams $X_1:=X[F_1/y_1,\,F_2/y_2]$, $X_2:=X_1[X(\tuple x_1)/\tuple x_1,\dots,X(\tuple x_n)/x_n\,]$, $Z:=\{s\in X_2\mid s(y_1)=s(y_2)\}$ and $Z':=\{s\in X_2\mid s(y_1)\neq s(y_2)\}$. Clearly now $Z\cup Z'=X_2$, $\mathcal{M}\true_{Z}y_1\!=\!y_2$ and $\mathcal{M}\true_{Z'}y_1\!\neq\!y_2$. We will then show that $Y[X(\tuple x_1)/\tuple x_1,\dots,X(\tuple x_n)/\tuple x_n]=Z\upharpoonright\dom(X)$.

Let $r\in Y[X(\tuple x_1)/\tuple x_1,\dots,X(\tuple x_n)/\tuple x_n]$. Now there is $s\in Y$ and tuples $\tuple a_1\in X(\tuple x_1),\dots,\tuple a_n\in X(\tuple x_n)$ s.t. $r=s[\tuple a_1/\tuple x_1,\dots,\tuple a_n/\tuple x_n\,]$. Since $s\in Y$, by the definition of $F_1$, we have $s[a/y_1]\in X[F_1/y_1]$. Let then $q:=s[a/y_1,a/y_2]$ and $q':=q[\tuple a_1/\tuple x_1,\dots,\tuple a_n/\tuple x_n\,]$, whence $q\in X_1$ and $q'\in X_2$. Since $q'(y_1)=q'(y_2)$, by the definition of $Z$, we have $q'\in Z$. But now $r=q'\upharpoonright\dom(X)\,\in\,Z\upharpoonright\dom(X)$, and thus we have shown that $Y[X(\tuple x_1)/\tuple x_1,\dots,X(\tuple x_n)/\tuple x_n]\subseteq Z\upharpoonright\dom(X)$.

Let then $r^*\in Z\upharpoonright\dom(X)$. Now there is $r\in Z$ s.t. $r^*=r\upharpoonright\dom(X)$. By the definition of $Z$ we have $r(y_1)=r(y_2)$. Since $r\in Z\subseteq X_2$, there is $q\in X_1$ and tuples $\tuple a_1\in X(\tuple x_1),\dots,\tuple a_n\in X(\tuple x_n)$ such that $r=q[\tuple a_1/\tuple x_1,\dots,\tuple a_n/\tuple x_n\,]$. Let $s:=q\upharpoonright\dom(X)$. Since $q(y_1)=q(y_2)$, by the definition of $F_1$, we must have $s\in Y$. Let $s':=s[\tuple a_1/\tuple x_1,\dots,\tuple a_n/\tuple x_n\,]$, whence $s'\in Y[X(\tuple x_1)/\tuple x_1,\dots,X(\tuple x_n)/\tuple x_n]$. But now $s' = r^*$ and thus $Z\upharpoonright\dom(X)\subseteq Y[X(\tuple x_1)/\tuple x_1,\dots,X(\tuple x_n)/\tuple x_n]$.

We have shown that $Y[X(\tuple x_1)/\tuple x_1,\dots,X(\tuple x_n)/\tuple x_n]=Z\upharpoonright\dom(X)$. Thus by the initial assumption we have $\mathcal{M}\true_{Z\upharpoonright\dom(X)}\varphi$ and thus by locality $\mathcal{M}\true_{Z}\varphi$. With similar argumentation we can show that $Y'[X(\tuple x_1)/\tuple x_1,\dots,X(\tuple x_n)/\tuple x_n]=Z'\upharpoonright\dom(X)$ and consequently $\mathcal{M}\true_{Z'}\psi$. 

Therefore it holds that $\mathcal{M}\true_{Z}y_1\!=\!y_2\wedge\varphi$ and $\mathcal{M}\true_{Z'}y_1\!\neq\!y_2\wedge\psi$. Furthermore we can conclude that $\mathcal{M}\true_{X}\Ee y_1\Ee y_2\uniform(\tuple x_1,\dots,\tuple x_n)\bigl((y_1\!=\!y_2\wedge\psi)\,\vee\,(y_1\!\neq\!y_2\wedge\theta)\bigr)$, i.e. $\mathcal{M}\true_X\,\varphi\!\!\underset{\scriptscriptstyle\tuple x_1,\dots,\tuple x_n}{\;\vee^\text{\bf U}}\!\!\psi$.
\end{proof}

\begin{remark}
We could easily modify the definition of unifying disjunction to make it work properly also in the case of single element models. Let
\[
	\varphi\!\!\underset{\scriptscriptstyle\tuple x_1,\dots,\tuple x_n}{\;\vee^\text{\bf U'}}\!\!\psi \; := \;
	\bigl(\Ae z_1\Ae z_2\,(z_1\!=\!z_2)\wedge (\varphi\vee\psi)\bigr) \; \sqcup \;
	\varphi\!\!\underset{\scriptscriptstyle\tuple x_1,\dots,\tuple x_n}{\;\vee^\text{\bf U}}\!\!\psi.
\]
It is easy to see that the truth condition given by Proposition~\ref{the: Unifying disjunction} holds for the operator above even without the extra assumption $\abs{M}>1$, as unifying disjunction becomes normal disjunction in the case of single element models. However, in this paper we are mainly using this operator as a tool in our main translation (Theorem~\ref{the: Translation from ESO to EXC}) where this simpler form suffices for our needs.
\end{remark}


\section{The expressive power of $\EXC[k]$}\label{sec: Expressive power}

In this section we analyze the expressive power of EXC[$k$] by comparing it with $k$-ary dependence logic and $k$-ary existential second order logic. Finally we discuss the correspondence between EXC[$k$] and INC[$k$].

Since the lax- and strict semantics are equivalent for exclusion logic, we may freely use either of them. In order to simplify some proofs in this section we decide to use the strict semantics for existential quantifier and lax-semantics for disjunction.\footnote{This combination is in some sense the simplest choice. It was used originally when dependence logic was defined (\cite{Vaananen07}). The lax- and strict-separation was noticed only after introducing logics that were not closed downwards.}

\subsection{Relationship between $\EXC$ and dependence logic}

Galliani \cite{Galliani12b} has shown that, without arity bounds, EXC is equivalent with dependence logic. However, if we consider the bounded arity fragments, this relationship becomes nontrivial. We first review Galliani's translation from exclusion logic to dependence logic (the translation is slightly simplified here).
\begin{proposition}[\cite{Galliani12b}]\label{the: Expressing exclusion atom}
Let $\tuple t_1,\tuple t_2$ be $k$-tuples of $L$-terms. The $k$-ary exclusion atom $\tuple t_1\mid\tuple t_2$  is logically equivalent to the depencende logic formula $\varphi$:
\begin{align*}
	\varphi := \Ae\tuple y\Ee w_1\!\Ee w_2\bigl(\dep(w_1)\!\wedge\dep(\tuple y,w_2)
	\wedge\left((w_1\!=\!w_2\wedge\tuple y\!\neq\!\tuple t_1)\vee(w_1\!\neq\!w_2\wedge\tuple y\!\neq\!\tuple t_2)\right)\bigr),
\end{align*}
where $\tuple y$ is a $k$-tuple of fresh variables and $w_1,w_2$ are fresh variables.
\end{proposition}
By inspecting Galliani's translations, we obtain the following result on the relationship between the arity fragments of exclusion logic and dependence logic.
\begin{corollary}\label{the: EXC vs. DEP}
The expressive power of \emph{EXC[$k$]} is in between $k$-ary dependence logic and \emph{($k$+$1$)}-ary dependence logic on the level of formulas.
\end{corollary}
\begin{proof}
By using the translation in Proposition~\ref{the: Expressing dependence atom} we can express $k$-ary dependence atoms with $k$-ary exclusion atoms. Moreover, by using the translation in Proposition~\ref{the: Expressing exclusion atom} we can express $k$-ary exclusion atoms with ($k$$+$$1$)-ary dependence atoms.
\end{proof}
By this result it is natural to ask whether these inclusions are proper, or whether EXC[$k$$+$$1$] collapses to some arity fragment of dependence logic. Let us inspect the special case $k=1$ with the following example.
\begin{example}[C.f. a similar example for $\INEX$ in \cite{Ronnholm15}]\label{ex: Properties of graphs}
Let $\mathcal{G}=(V,E)$ be an undirected graph. Now we have
\begin{enumerate}
\item[(a)] $\mathcal{G}$ is disconnected if and only if
\begin{align*}
	\mathcal{G}\true\Ae z\Ee x_1\Ee x_2\,\bigl((x_1\!=\!z\vee x_2\!=\!z)\wedge x_1&\,|\, x_2
	\wedge(\Ae y_1\inc x_1)(\Ae y_2\inc x_2)\neg Ey_1y_2\bigr).
\end{align*}
\item[(b)] $\mathcal{G}$ is $k$-colorable if and only if
\begin{align*}
	\mathcal{G}\true \gamma_{\leq k}\;\sqcup\;
	\Ae z\Ee x_1\dots\Ee x_k\,\Bigl(&\bigvee_{i\leq k}\!x_i\!=\!z\,\wedge\bigwedge_{i\neq j}x_i\,|\, x_j\, \\[-0,1cm]
	&\quad\wedge\bigwedge_{i\leq k}\!(\Ae y_1\inc x_i)(\Ae y_2\inc x_i)\neg Ey_1y_2\Bigr),
\end{align*}
where $\gamma_{\leq k}:= \Ee x_1\dots\Ee x_k\Ae y\,\bigl(\bigvee\limits_{i\leq k}y=x_i\bigr)$.
\end{enumerate}

We explain briefly why these equivalences hold. In (a), suppose that the given sentence is true in $\mathcal{G}$. Let $X$ be the team after the quantification of $z$, $x_1$ and $x_2$. Since we have $\mathcal{M}\true_X x_1\!=\!z \vee x_2\!=\!z$ and $X(z)=V$, it must be that $X(x_1)\cup X(x_2)=V$. And since $\mathcal{M}\true_X x_1\,|\, x_2$, it must be that $X(x_1)\cap X(x_2)=\emptyset$. Hence the sets $X(x_1)$ and $X(x_2)$ must form a disjoint union of all vertices. Because $\mathcal{M}\true_X (\Ae y_1\subseteq x_1)(\Ae y_2\subseteq x_2)\neg Ey_1y_2$, we have $(a,b)\notin E$ for any pair $(a,b)$ in $X(x_1)\times X(x_2)$. That is, there is no edge between these disjoint sets and thus $\mathcal{G}$ must be disconnected. It is easy to see that also the converse claim holds.

Let us then consider the equivalence in (b). If $\mathcal{G}\true\gamma_{\leq k}$ the graph is trivially $k$-colorable. Else let $X$ be the team after the quantification of variables $z,x_1,\dots,x_k$. As above, the truth of $\bigvee_{i\leq k}x_i\!=\!z$ guarantees that $\bigcup_{i\leq k} X(x_i)=V$ and the truth of exclusion atoms guarantees that sets $X(x_i)$ are disjoint. Let these sets be the coloring of the graph. Because we have for all $i\leq n$: $\mathcal{M}\true_X (\Ae y_1\subseteq x_i)(\Ae y_2\subseteq x_i)\neg Ey_1y_2$, it follows that $(a,b)\notin E$ for any pair $a,b\in X_i$ and $i\leq n$. That is, there is no edge between any two vertices chosen from a single color set, i.e. the coloring is correct. It is easy to see that also the converse claim holds.
\end{example}
\begin{corollary}
The expressive power of $\EXC[1]$ is properly in between $1$-ary and $2$-ary dependence logics, on the level of both sentences and formulas.
\end{corollary}
\begin{proof}
By Corollary \ref{the: EXC vs. DEP}, the expressive power of EXC[$1$] is in between $1$-ary and $2$-ary dependence logics. By the results of Galliani \cite{Galliani12b}, $1$-ary dependence logic is not stronger than FO on the level of sentences. However, by Example~\ref{ex: Properties of graphs}, there are sentences of EXC[$1$] that cannot be expressed in FO. Thus EXC[$1$] is strictly stronger than $1$-ary dependence logic on the level of sentences.

On the other hand, there are properties that are definable $2$-ary dependence logic, but which cannot be expressed in \emph{existential monadic second order logic}, EMSO, such as infinity of a model and even cardinality (\cite{Vaananen07}). But since INEX[$1$] is equivalent to EMSO on the level of sentences (\cite{Ronnholm15}), EXC[$1$] must be strictly weaker than $2$-ary dependence logic on the level of sentences. 
\end{proof}

\subsection{Capturing the arity fragments of $\ESO$ with $\EXC$}\label{ssec: Capturing ESO[k]}

In this subsection we will compare the expressive power of EXC with \emph{existential second order logic}, ESO. We denote the $k$-ary fragment of ESO (where at most $k$-ary relation symbols can be quantified) by ESO[$k$]. We will formulate a translation from ESO[$k$] to EXC[$k$] on the level of sentences by using the idea from the following observation: Suppose that $X$ is a team and $\tuple x$, $\tuple w$, $\tuple w^c$ are tuples variables s.t. $X(\tuple w^c)=\overline{X(\tuple w)}$. Now $\mathcal{M}\true_X \tuple x\subseteq\tuple w$ iff $\mathcal{M}\true_X\tuple x\mid\tuple w^c$.

In our translation from ESO[$k$] to INEX[$k$] (\cite{Ronnholm15}) the quantified $k$-ary relation symbols $P_i$ of an $\ESOset$-formula were simply replaced with $k$-tuples $\tuple w_i$ of quantified \emph{first order} variables. Then the formulas of the form $P_i\tuple t$ were replaced with the inclusion atoms $\tuple t\subseteq\tuple w_i$ and the formulas of the form $\neg P_i\tuple t$ with the exclusion atoms $\tuple t\;|\,\tuple w_i$. In order to eliminate inclusion atoms from this translation, we also need to quantify a tuple $\tuple w_i^c$ of variables for each $P_i$ and set a requirement that $\tuple w_i^c$ must be given complementary values with respect to the values of~$\tuple w_i$. This requirement is possible to be set in exclusion logic if we are restricted to sentences. Then we simply replace inclusion atoms $\tuple t\inc\tuple w_i$ with the corresponding exclusion atoms $\tuple t\;|\,\tuple w_i^c$.

We also need to consider the quantification of the empty set and the full relation $M^k$ as special cases. This is because tuples $\tuple w_i$ and also their ``complements'' $\tuple w_i^c$ must always be given a nonempty set of values. For this we use special ``label variables'' $w_i^\circ$ and $w_i^\bullet$ for each relation symbol $P_i$. We first quantify some \emph{constant} value for a variable $u$. Then we can give the value of $u$ for $w_i^\circ$ to ``announce'' the quantification of the empty set or analogously we can give it for $w_i^\bullet$ to announce the quantification of the full relation. In order to give these label values, there must be at least two elements in the model. For handling the special case of single element models we will use the following easy lemma (we omit the proof).
\begin{lemma}\label{the: Single element}
Let $\varphi$ be an $\ESOset$-sentence. Now there exists an $\FOset$-sentence $\chi$, such that we have $\mathcal{M}\true\varphi$ iff $\mathcal{M}\true\chi$, for all $L$-models $\mathcal{M}=(M,\mathcal{I})$ for which $\abs{M}=1$.
\end{lemma}
The remaining problem is that in the translation from ESO to INEX we also needed a new connective called \emph{term value preserving disjunction} (\cite{Ronnholm15}) to avoid the ``loss of information'' on the values of variables $\tuple w_i$ when evaluating disjunctions (as after splitting the team, there might be less values for some variables in the subteams).
%
%
This time we can use unifying disjunction instead to avoid the loss of information on the values of both the tuples $\tuple w_i$ and the tuples $\tuple w_i^c$. We are now ready to formulate our main theorem.

\begin{theorem}\label{the: Translation from ESO to EXC}
For every $\ESOset[k]$-sentence $\Phi$ there exists an $\EXCset[k]$-sentence $\varphi$ such that 
\[
	\mathcal{M}\true\varphi \;\text{ iff }\; \mathcal{M}\true\Phi.
\]
\end{theorem}
\begin{proof}
Since $\Phi$ is an $\ESOset[k]$-sentence, there exists a $\FOset$-sentence $\delta$ and relation symbols $P_1,\dots,P_n$ so that $\Phi=\Ee P_1\dots\Ee P_n\delta$. Without losing generality, we may assume that $P_1,\dots,P_n$ are all $k$-ary. Let $\tuple w_1,\dots,\tuple w_n$ and $\tuple w_1^c,\dots,\tuple w_n^c$ be $k$-tuples of variables and $w_1^\circ,\dots,w_n^\circ,w_1^\bullet,\dots,w_n^\bullet$ and $u$ be variables such that all of these variables are distinct and do not occur in the sentence $\delta$. 

\medskip

\noindent
Let $\psi\in\subf(\delta)$. The formula $\psi'$ is defined recursively:
\begin{align*}
	\psi' &= \psi \quad\text{ if $\psi$ is a literal and $P_i$ does not occur in } \psi \text{ for any } i\leq n \\
	(P_i \tuple t\,)' &= (\tuple t\mid\!\tuple w_i^c\vee w_i^\bullet\!=\!u)\wedge w_i^\circ\!\neq\!u
 	\quad\text{ for all } i\leq n \\
	(\neg P_i \tuple t\,)' &= (\tuple t\mid\!\tuple w_i\vee w_i^\circ\!=\!u)\wedge w_i^\bullet\!\neq\!u
	 \quad\text{ for all } i\leq n\\
	(\psi\wedge\theta)' &= \psi'\!\wedge\theta' \\
	(\psi\vee\theta)' &= \psi'\veebar^\text{\bf U}\theta', \quad\;\; \text{ where }\, 
	\veebar^\text{\bf U} := 
	\underset{\scriptscriptstyle\tuple w_1,\dots,\tuple w_n,\tuple w_1^c,\dots,\tuple w_n^c}{\;\vee^\text{\bf U}}\\[-0,2cm]
	(\Ee x\,\psi)' &= \Ee x\,\psi' \\
	(\Ae x\,\psi)' &= \Ae x\,\psi'.
\end{align*}

Let $\chi$ be a $\FOset$-sentence determined by the Lemma \ref{the: Single element} for the sentence $\Phi$ and let $\tuple z$ be a $k$-tuple of fresh variables. Let $\gamma_{=1}$ be an abbreviation for the sentence $\Ae z_1\Ae z_2\,(z_1\!=\!z_2)$. Now we can define the sentence $\varphi$ in the following way:
\begin{align*}
	\varphi \,:=\, \;&(\gamma_{=1}\wedge\chi)\,\sqcup\,\Ee u
	\Ee w_1^\circ\dots\Ee w_n^\circ\Ee w_1^\bullet\dots\Ee w_n^\bullet \\
	&\hspace{2,1cm}\Ae\tuple z\,\Ee\tuple w_1\dots\Ee\tuple w_n\Ee\tuple w_1^c\dots\Ee\tuple w_n^c
	\bigl(\bigwedge_{i\leq n}(\tuple z=\tuple w_i\vee\tuple z=\tuple w_i^c)\wedge\delta'\bigr).
\end{align*}

\noindent
Clearly $\varphi$ is an $\EXCset[k]$-sentence. 

\begin{remark}
Since we are using the tuples $\vec w_i$ and $\vec w_i^c$ to simulate a quantified relation and its complement, respectively, it would be natural to add the requirement $\bigwedge_{i\leq n}\tuple w_i\exc\tuple w_i^c$ to the sentence $\varphi$ above. However, we will see that this is not necessary, since it suffices that $\vec w_i$ and $\vec w_i^c$ are quantified in such a way that $X(\tuple w_i)\cup X(\tuple w_i^c)=M^k$ in the resulting team $X$. This condition is achieved by first universally quantifying a tuple $\tuple z$ and adding disjunction $\tuple z=\tuple w_i\vee\tuple z=\tuple w_i^c$ for each $i\leq n$ (compare with a similar idea in the sentences of Example~\ref{ex: Properties of graphs}).
\end{remark}

\noindent
We write 
\[
V^* := \vr(u w_1^\circ \dots w_n^\circ w_1^\bullet\dots w_n^\bullet\tuple w_1\dots\tuple w_n\tuple w_1^c\dots\tuple w_n^c).
\]
Before proving the claim of this theorem, we prove the following two claims.

\begin{claim}\label{R1}
Let $\mathcal{M}$ be an $L$-model with at least two elements. Let $\mu\in\subf(\delta)$ and let $X$ a team for which $V^*\!\subseteq\!\dom(X)$ and the following assumptions hold:
\[
	\begin{cases}
		X(\tuple w_i)\cup X(\tuple w_i^c)=M^k\;\text{ for each } i\leq n. \\
		\text{The values of } w_i^\circ,w_i^\bullet \;(i\leq n) \text{ and } u \text{ are constants in } X.
	\end{cases}
\]
Let $\mathcal{M}' := \mathcal{M}[\tuple A/\tuple P]\;\; (=\mathcal{M}[A_1/P_1,\dots,A_n/P_n])$, where
\begin{align*}
	A_i &=
	\begin{cases}
		\emptyset \qquad\,\text{ if } X(w_i^\circ)=X(u) \text{ and }\, X(w_i^\bullet)\neq X(u) \\
		M^k \quad\,\text{ if } X(w_i^\bullet)=X(u) \text{ and }\, X(w_i^\circ)\neq X(u) \\
		X(\tuple w_i) \text{ else}.
	\end{cases}
\end{align*}

\noindent
Now the following implication holds: 
\[
	\text{If } \mathcal{M}\true_X\mu', \text{ then } \mathcal{M}'\true_X\mu.
\]
\end{claim}
\noindent
We prove this claim by structural induction on $\mu$:
\begin{itemize}
\item If $\mu$ is a literal and $P_i$ does not occur in $\mu$ for any $i\leq n$, then the claim holds trivially since $\mu'=\mu$.

\smallskip

\item Let $\mu=P_j\tuple t$ for some $j\leq n$. 
Suppose that we have $\mathcal{M}\true_X (P_j\tuple t\,)'$, i.e. $\mathcal{M}\true_X(\tuple t\mid\!\tuple w_j^c\vee w_j^\bullet\!=\!u)\wedge w_j^\circ\!\neq\!u$. Because the values of $u$, $w_j^\circ$ are constants in $X$ and $\mathcal{M}\true_X w_j^\circ\!\neq\!u$, we have $X(w_j^\circ)\neq X(u)$. If $X(w_j^\bullet)=X(u)$, then $A_j=M^k$ and thus trivially $\mathcal{M}'\true_X P_j\tuple t$. 
Suppose then that $X(w_j^\bullet)\neq X(u)$ whence $A_j=X(\tuple w_j)$. Because the values of  $u$, $w_j^\bullet$ are constants in $X$ and $\mathcal{M}\true_X\tuple t\mid\!\tuple w_j^c\vee w_j^\bullet\!=\!u$, it must be that $\mathcal{M}\true_X\tuple t\mid\!\tuple w_j^c$. Now $X(\tuple t\,)\cap X(\tuple w_j^c)=\emptyset$ and $X(\tuple w_j)\cup X(\tuple w_j^c)=M^k$. Hence $X(\tuple t\,)\subseteq\overline{X(\tuple w_j^c)}\subseteq X(\tuple w_j)=A_j$ and thus $\mathcal{M}'\true_X P_j\tuple t$.


\item Let $\mu=\neg P_j\tuple t$ for some $j\leq n$. 
Suppose that we have $\mathcal{M}\true_X (\neg P_j\tuple t\,)'$, i.e. $\mathcal{M}\true_X(\tuple t\mid\!\tuple w_j\vee w_j^\circ\!=\!u)\wedge w_j^\bullet\!\neq\!u$. Because the values of $u$, $w_j^\bullet$ are constants and $\mathcal{M}\true_X w_j^\bullet\!\neq\!u$, we have $X(w_j^\bullet)\neq X(u)$. If $X(w_j^\circ)=X(u)$, then $A_j=\emptyset$ and thus trivially $\mathcal{M}'\true_X \neg P_j\tuple t$. 
Suppose then that $X(w_i^\circ)\neq X(u)$ whence $A_j=X(\tuple w_j)$. Because the values of $u$, $w_j^\circ$ are constants in $X$ and $\mathcal{M}\true_X\tuple t\mid\!\tuple w_j\vee w_j^\circ\!=\!u$, we have $\mathcal{M}\true_X\tuple t\mid\!\tuple w_j$. Now $X(\tuple t\,)\subseteq\overline{X(\tuple w_j)}=\overline{A_j}$ and thus $\mathcal{M}'\true_X \neg P_j\tuple t$.


\item The case $\mu=\psi\wedge\theta$ is straightforward to prove.


\item Let $\mu=\psi\vee\theta$. 
Suppose that $\mathcal{M}\true_X(\psi\vee\theta)'$, i.e. $\mathcal{M}\true_X\psi'\veebar^\text{\bf U}\theta'$. By Proposition \ref{the: Unifying disjunction} there exist $Y_1,Y_2\subseteq X$ s.t. $Y_1\cup Y_2 = X$, $\mathcal{M}\true_{Y_1^*}\psi'$ and $\mathcal{M}\true_{Y_2^*}\theta'$, where
\[
	\begin{cases}
		Y_1^* := Y_1[X(\tuple w_1)/\tuple w_1,\dots,X(\tuple w_n)/\tuple w_n,
				X(\tuple w_1^c)/\tuple w_1^c,\dots,X(\tuple w_n^c)/\tuple w_n^c] \\
		Y_2^* := Y_2[X(\tuple w_1)/\tuple w_1,\dots,X(\tuple w_n)/\tuple w_n,
				X(\tuple w_1^c)/\tuple w_1^c,\dots,X(\tuple w_n^c)/\tuple w_n^c].
	\end{cases}
\]
Now the sets of values for $\tuple w_i$ and $\tuple w_i^c$ are the same in $Y_1^*$ and $Y_2^*$ as in $X$. Because the values of $u$ and $w_i^\circ$, $w_i^\bullet$ are constants in $X$ they have (the same) constant values in $Y_1^*$ and $Y_2^*$. Hence, by the inductive hypothesis, we have $\mathcal{M}'\true_{Y_1^*}\psi$ and $\mathcal{M}'\true_{Y_2^*}\theta$. Since none of the variables in $V^*$ occurs in $\psi\vee\theta$, by locality $\mathcal{M}'\true_{Y_1}\psi$ and $\mathcal{M}'\true_{Y_2}\theta$. Therefore $\mathcal{M}'\true_X\psi\vee\theta$.


\item The cases $\mu=\Ee x\,\psi$ and $\mu=\Ae x\,\psi$ are straightforward to prove. 
(Note here that, since $x\notin V^*$, the assumptions of Claim~\ref{R1} hold in the resulting team also after the quantification of $x$.) 
\end{itemize}

\medskip

\begin{claim}\label{R2}
Let $\mathcal{M}$ be an $L$-model with at least two elements. Let $\mu\in\subf(\delta)$ and $X$ be a team such that $\dom(X)=\fr(\mu)$. Assume that $A_1,\dots,A_n\subseteq M^k$, $\mathcal{M}' := \mathcal{M}[\tuple A/\tuple P\,]$ and $a,b\in M$ s.t. $a\neq b$. Let
\begin{align*}
	X' &:= X\bigl[\{a\}/u,B_1^\circ/w_1^\circ,\dots,B_n^\circ/w_n^\circ,
	B_1^\bullet/w_1^\bullet,\dots,B_n^\bullet/w_n^\bullet, \\[-0,1cm]
	&\hspace{2,1cm}B_1/\tuple w_1,\dots,B_n/\tuple w_n,B_1^c/\tuple w_1^c,\dots,B_n^c/\tuple w_n^c\bigr],\\
	&\hspace{0,9cm}\text{ where }\,
	\begin{cases}
		B_i^\circ = \{a\},\; B_i^\bullet = \{b\} \text{ and } B_i = B_i^c = M^k \qquad\;\, \text{ if } A_i=\emptyset \\
		B_i^\circ = \{b\},\; B_i^\bullet = \{a\} \text{ and } B_i = B_i^c = M^k \qquad\;\, \text{ if } A_i=M^k \\
		B_i^\circ = \{b\},\; B_i^\bullet = \{b\}, \; B_i = A_i \text{ and } B_i^c = \overline{A_i} \quad\: \text{ else}.
	\end{cases}
\end{align*}
Now the following implication holds: 
\[
	\text{If } \mathcal{M}'\true_{X}\mu, \text{ then } \mathcal{M}\true_{X'}\mu'.
\]
\end{claim}
\noindent
We prove this claim by structural induction on $\mu$. Note that if $X=\emptyset$, then also $X'=\emptyset$ and thus the claim holds by the empty team property. Hence we may assume that $X\neq\emptyset$.
\smallskip
\begin{itemize}
\item If $\mu$ is a literal and $P_i$ does not occur in $\mu$ for any $i\leq n$, then the claim holds by locality since $\mu'=\mu$.


\item Let $\mu=P_j\tuple t$ for some $j\leq n$. 
Suppose $\mathcal{M}'\true_X P_j\tuple t$, i.e. $X(\tuple t\,)\subseteq P_j^\mathcal{M'}=A_j$. Since $X\neq\emptyset$, also $X(\tuple t\,)\neq\emptyset$ and thus $A_j\neq\emptyset$. Hence $X'(w_j^\circ)=\{b\}$, and thus $\mathcal{M}\true_{X'}w_i^\circ\!\neq\!u$ since $X'(u)=\{a\}$ . If $A_j=M^k$, then $X'(w_i^\bullet)=\{a\}$ and thus $\mathcal{M}\true_{X'}w_j^\bullet=u$, whence $\mathcal{M}\true_{X'}(\tuple t\mid\!\tuple w_j^c\vee w_j^\bullet\!=\!u)\wedge w_j^\circ\!\neq\!u$, i.e. $\mathcal{M}\true_{X'}(P_j\tuple t\,)'$.
Suppose then that $A_j\neq M^k$. Now we have $X'(\tuple w_j^c)=\overline{A_j}$, i.e. $\overline{X'(w_j^c)}=A_j$, and thus $X'(\tuple t\,)\!=\!X(\tuple t\,)\subseteq A_j\!=\!\overline{X'(\tuple w_j^c)}$. Hence we have $\mathcal{M}\true_{X'} \tuple t\mid\!\tuple w_j^c$ and thus $\mathcal{M}\true_{X'}(\tuple t\mid\!\tuple w_j^c\vee w_j^\bullet\!=\!u)\wedge w_j^\circ\!\neq\!u$, i.e. $\mathcal{M}\true_{X'}(P_j\tuple t\,)'$.


\item Let $\mu=\neg P_j\tuple t$ for some $j\leq n$. 
Suppose that we have $\mathcal{M}'\true_X\neg P_j\tuple t$, i.e. $X(\tuple t\,)\subseteq \overline{P_j^\mathcal{M'}}=\overline{A_j}$. Since $X\neq\emptyset$, we have $X(\tuple t\,)\neq\emptyset$ and thus $\overline{A_j}\neq\emptyset$, i.e. $A_j\neq M^k$.  Hence $X'(w_j^\bullet)=\{b\}$, and thus $\mathcal{M}\true_{X'}w_i^\bullet\!\neq\!u$ since $X'(u)=\{a\}$. If $A_j=\emptyset$, then $X'(w_i^\circ)=\{a\}$ and thus $\mathcal{M}\true_{X'}w_j^\circ=u$, whence $\mathcal{M}\true_{X'}(\tuple t\mid\!\tuple w_j\vee w_j^\circ\!=\!u)\wedge w_j^\bullet\!\neq\!u$, i.e. $\mathcal{M}\true_{X'}(\neg P_j\tuple t\,)'$.
Suppose then that we have $A_j\neq\emptyset$. Then $X'(\tuple w_j)=A_j$ and thus it holds that $X'(\tuple t\,)\!=\!X(\tuple t\,)\subseteq \overline{A_j}\!=\!\overline{X'(\vec w_j)}$. Hence we have $\mathcal{M}\true_{X'} \tuple t\mid\!\tuple w_j$ and therefore $\mathcal{M}\true_{X'}(\tuple t\mid\!\tuple w_j\vee w_j^\circ\!=\!u)\wedge w_j^\bullet\!\neq\!u$, i.e. $\mathcal{M}\true_{X'}(\neg P_j\tuple t\,)'$.


\item The case $\mu=\psi\wedge\theta$ is straightforward to prove.


\item Let $\mu=\psi\vee\theta$. 
Suppose that $\mathcal{M}'\true_X\psi\vee\theta$, i.e. there exist $Y_1,Y_2\subseteq X$ s.t. $Y_1\cup Y_2=X$, $\mathcal{M}'\true_{Y_1}\psi$ and $\mathcal{M}'\true_{Y_2}\theta$. Let $Y_1',Y_2'$ be the teams obtained by extending the teams $Y_1,Y_2$ as $X'$ is obtained by extending $X$. Then, by the inductive hypothesis, we have $\mathcal{M}\true_{Y_1'}\psi'$ and $\mathcal{M}\true_{Y_2'}\theta'$. Now the following holds:
\[
	\begin{cases}
		Y_1' = Y_1'[X'(\tuple w_1)/\tuple w_1,\dots,X'(\tuple w_n)/\tuple w_n,
				X'(\tuple w_1^c)/\tuple w_1^c,\dots,X'(\tuple w_n^c)/\tuple w_n^c] \\
		Y_2' = Y_2'[X'(\tuple w_1)/\tuple w_1,\dots,X'(\tuple w_n)/\tuple w_n,
				X'(\tuple w_1^c)/\tuple w_1^c,\dots,X'(\tuple w_n^c)/\tuple w_n^c].
	\end{cases}
\]
Note that also $Y_1',Y_2'\subseteq X'$ and $Y_1'\cup Y_2'=X'$. Thus by Proposition \ref{the: Unifying disjunction} $\mathcal{M}\true_{X'}\psi'\veebar^\text{\bf U}\theta'$, i.e. $\mathcal{M}\true_{X'}(\psi\vee\theta)'$.


\item Let $\mu=\Ee x\,\psi$ \;(the case $\mu=\Ae x\,\psi$ is proven similarly).
Suppose that $\mathcal{M}'\true_X\Ee x\,\psi$, i.e. there exists $F:X\rightarrow M$ s.t. $\mathcal{M}'\true_{X[F/x]}\psi$. Let $F':X'\rightarrow M$ such that $s\mapsto F(s\upharpoonright\fr(\mu))$ for each $s\in X'$. Note that $F'$ is well defined since $\dom(X)=\fr(\mu)$ by the assumption. 

Let $(X[F/x])'$ be a team that is obtained by extending the team $X[F/x]$ analogously as $X'$ is obtained by extending $X$. Now by inductive hypothesis we have $\mathcal{M}\true_{(X[F/x])'}\psi'$. By the definition of $F'$ it is easy to see that $(X[F/x])'=X'[F'/x]$ and thus $\mathcal{M}\true_{X'[F'/x]}\psi'$. Hence we have $\mathcal{M}\true_{X'}\Ee x\,\psi'$, i.e. $\mathcal{M}\true_{X'}(\Ee x\,\psi)'$.
\end{itemize}

\medskip

We are now ready to prove the claim of this theorem:
\[
	\mathcal{M}\true\varphi \; \text{ iff } \; \mathcal{M}\true\Phi.
\]
Suppose first that $\mathcal{M}\true\varphi$, i.e. $\mathcal{M}\true\gamma_{=1}\wedge\chi$ or
\begin{align*}
	&\mathcal{M}\true\Ee u\Ee w_1^\circ\dots\Ee w_n^\circ\Ee w_1^\bullet\dots\Ee w_n^\bullet \tag{$\star$} \\[-0,1cm]
	&\hspace{1,5cm}\Ae\tuple z\,\Ee\tuple w_1\dots\Ee\tuple w_n\Ee\tuple w_1^c\dots\Ee\tuple w_n^c
	\bigl(\bigwedge_{i\leq n}(\tuple z=\tuple w_i\vee\tuple z=\tuple w_i^c)\wedge\delta'\bigr).
\end{align*}
If $\mathcal{M}\true\gamma_{=1}\wedge\chi$, the claim holds by Lemma \ref{the: Single element}. Suppose then ($\star$), whence by the (strict) semantics of existential quantifier there are $a,b_1\dots b_n,b_1',\dots,b_n'\!\in\!M$ such that
\[
	\mathcal{M}\true_{X_1}\Ae\tuple z\,\Ee\tuple w_1\dots\Ee\tuple w_n\Ee\tuple w_1^c\dots\Ee\tuple w_n^c
	\bigl(\bigwedge_{i\leq n}(\tuple z=\tuple w_i\vee\tuple z=\tuple w_i^c)\wedge\delta'\bigr),
\]
where $X_1:=\{\emptyset[a/u,b_1/w_1^\circ,\dots,b_n/w_n^\circ,b_1'/w_1^\bullet,\dots,b_n'/w_n^\bullet]\}$. 
Note that since $X_1$ consists only of a single assignment, the values of $u$, $w_i^\circ$ and $w_i^\bullet$ ($i\leq n$) are trivially constants in the team $X_1$. 
Let $X_2:=X_1[M^k/\tuple z\,]$. Now there exist functions $\mathcal{F}_i:X_2[\mathcal{F}_1/\tuple w_1,\dots,\mathcal{F}_{i-1}/\tuple w_{i-1}]\rightarrow M^k$ such that
\[
	\mathcal{M}\true_{X_3}\Ee\tuple w_1^c\dots\Ee\tuple w_n^c
	\bigl(\bigwedge_{i\leq n}(\tuple z=\tuple w_i\vee\tuple z=\tuple w_i^c)\wedge\delta'\bigr),
\]
where $X_3:=X_2[\mathcal{F}_1/\tuple w_1,\dots,\mathcal{F}_n/\tuple w_n]$.

Furthermore there exist functions $\mathcal{F}_i':X_3[\mathcal{F}_1'/\tuple w_1^c,\dots,\mathcal{F}_{i-1}'/\tuple w_{i-1}^c]\rightarrow M^k$ such that $\mathcal{M}\true_{X_4}\bigwedge_{i\leq n}(\tuple z=\tuple w_i\vee\tuple z=\tuple w_i^c)\wedge\delta'$, where $X_4:=X_3[\mathcal{F}_1'/\tuple w_1^c,\dots,\mathcal{F}_n'/\tuple w_n^c]$. Since $X_4(\tuple z)=M^k$ and $\mathcal{M}\true_{X_4}\bigwedge_{i\leq n}(\tuple z=\tuple w_i\vee\tuple z=\tuple w_i^c)$, it is easy to see that $X_4(\tuple w_i)\cup X_4(\tuple w_i^c)=M^k$ for each $i\leq n$. Now all the assumptions of Claim \ref{R1} hold for the team $X_4$. Let $\mathcal{M}' := \mathcal{M}[\tuple A/\tuple P]$, where
\begin{align*}
	A_i &=
	\begin{cases}
		\emptyset \qquad\;\;\text{ if } X_4(w_i^\circ)=X_4(u) \text{ and }\, X_4(w_i^\bullet)\neq X_4(u) \\
		M^k \quad\;\;\text{ if } X_4(w_i^\bullet)=X_4(u) \text{ and }\, X_4(w_i^\circ)\neq X_4(u) \\
		X_4(\tuple w_i) \text{ else}.
	\end{cases}
\end{align*}
Since $\mathcal{M}\true_{X_4}\delta'$, by Claim \ref{R1} we have $\mathcal{M'}\true_{X_4}\delta$. By locality $\mathcal{M'}\true\delta$, and therefore $\mathcal{M}\true\Phi$.

\bigskip

Suppose then that $\mathcal{M}\true\Phi$. If $\abs{M}=1$, then by Lemma \ref{the: Single element} we have $\mathcal{M}\true\gamma_{=1}\wedge\chi$ and thus $\mathcal{M}\true\varphi$. Hence we may assume that $\abs{M}\geq 2$, whence there exist $a,b\in M$ s.t. $a\neq b$. Since $\mathcal{M}\true\Phi$, there exist $A_1,\dots,A_n\subseteq M^k$ s.t. $\mathcal{M}[\tuple A/\tuple P]\true\delta$. Let
\begin{align*}
	X' := \{\emptyset\}\bigl[\{a\}/u,\,&B_1^\circ/w_1^\circ,\dots,B_n^\circ/w_n^\circ,
	B_1^\bullet/w_1^\bullet,\dots,B_n^\bullet/w_n^\bullet, \\
	&B_1/\tuple w_1,\dots,B_n/\tuple w_n,B_1^c/\tuple w_1^c,\dots,B_n^c/\tuple w_n^c\bigr],
\end{align*}
where $B_i^\circ, B_i^\bullet, B_i, B_i^c$ $(i\leq n)$ are defined as in the assumptions of Claim \ref{R2}.
Since $\mathcal{M}[\tuple A/\tuple P]\true\delta$, by Claim \ref{R2} we have $\mathcal{M}\true_{X'}\delta'$. Let
\begin{align*}
	\mathcal{F}:\{\emptyset\}\rightarrow &M^{2n+1}, \;\;\;\emptyset\mapsto a b_1\dots b_n b_1'\dots b_n', \\
	&\text{ where }\quad
	\begin{cases}
		b_i = a\, \text{ if } A_i =\emptyset \\
		b_i = b\; \text{ else }
	\end{cases}
	\text{ and }\quad
	\begin{cases}
		b_i' = a\, \text{ if } A_i =M^k \\
		b_i' = b\; \text{ else. }
	\end{cases}
\end{align*}
Let $X_1:=\{\emptyset\}[\mathcal{F}/u w_1^\circ\dots w_n^\circ w_1^\bullet\dots w_n^\bullet]$ and let $X_2:=X_1[M^k/\tuple z\,]$. We fix some $\tuple b_i\in A_i$ for each $i\leq n$ for which $A_i\neq\emptyset$ and define the functions
\begin{align*}
	\mathcal{F}_i:X_2[\mathcal{F}_1/\tuple w_1,\dots,\mathcal{F}_{i-1}/\tuple w_{i-1}]\rightarrow M^k,
	\quad\begin{cases}
		s\mapsto s(\tuple z) \;\text{ if } s(\tuple z)\in A_i \text{ or } A_i=\emptyset \\
		s\mapsto \tuple b_i \;\quad\text{ else.}
	\end{cases} 
\end{align*}
Let $X_3:=X_2[\mathcal{F}_1/\tuple w_1,\dots,\mathcal{F}_n/\tuple w_n]$. We fix some $\tuple b_i'\in\overline{A_i}$ for each $i\leq n$ for which $A_i\neq M^k$ and define
\begin{align*}
	\mathcal{F}_i':X_3[\mathcal{F}_1'/\tuple w_1^c,\dots,\mathcal{F}_{i-1}'/\tuple w_{i-1}^c]\rightarrow M^k,
	\;\begin{cases}
		s\mapsto s(\tuple z) \;\text{ if } s(\tuple z)\in\overline{A_i} \text{ or } A_i=M^k \\
		s\mapsto \tuple b_i' \;\quad\text{ else.}
	\end{cases} 
\end{align*}
\noindent
Let $X_4:=X_3[\mathcal{F}_1'/\tuple w_1^c,\dots,\mathcal{F}_n'/\tuple w_n^c]$. By the definitions of the functions $\mathcal{F}_i, \mathcal{F}_i'$ it is quite easy to see that $\mathcal{M}\true_{X_4}\bigwedge_{i\leq n}(\tuple z=\tuple w_i\vee\tuple z=\tuple w_i^c)$.
By the definitions of the choice functions for the variables in $V^*$, we observe that $X_4\upharpoonright V^*\subseteq X'$ (note here that the variables in $\tuple z$ are not in $\dom(X')$). 
Hence by locality and downwards closure $\mathcal{M}\true_{X_4}\delta'$. Thus $\mathcal{M}\true_{X_4}\bigwedge_{i\leq n}(\tuple z=\tuple w_i\vee\tuple z=\tuple w_i^c)\wedge\delta'$ and furthermore $\mathcal{M}\true\varphi$.
\end{proof}
\begin{corollary}\label{the: EXC vs. ESO}
On the level of sentences \,$\EXC[k] \equiv \ESO[k]$. 
\end{corollary}

\begin{proof}
In \cite{Ronnholm15} we have presented a translation from EXC[$k$] to ESO[$k$]. By Theorem~\ref{the: Translation from ESO to EXC}, on the level of sentences, there is also a translation from ESO[$k$] to EXC[$k$].
\end{proof}

In particular, we can capture existential monadic second order logic, EMSO, by using \emph{unary} exclusion atoms. This is particularly interesting since EMSO cannot be captured with any arity fragment of dependence nor independence logic (as a consequence by results in \cite{Durand12,Kontinen13a}). Hence we argue that exclusion logic deserves extra recognition by capturing this important fragment of ESO.


\subsection{Relationship between $\INC[k]$ and $\EXC[k]$}\label{ssec: INC & EXC}

Since by \cite{Ronnholm15} INEX[$k$] captures ESO[$k$], by Corollary~\ref{the: EXC vs. ESO} we can deduce that INEX[$k$] $\equiv$ EXC[$k$] on the level of sentences. Hence, on the level of sentences, $k$-ary inclusion atoms do not increase the expressive power of EXC[$k$].

By Dawar~\cite{Dawar98}, 3-colorability of a graph cannot be expressed in fixed point logic. Since by \cite{Hella13} INC is equivalent with positive greatest fixed point logic, this property is not expressible in INC. However, since it can be expressed in EXC[$1$] (recall Example~\ref{ex: Properties of graphs}), INC[$k$] is \emph{strictly} weaker than EXC[$k$] on the level of sentences for any $k$. 

\begin{corollary}\label{the: INC & EXC}
On the level of sentences \emph{INC}$[k]$ $<$ \emph{EXC}$[k]$ for any $k\geq 1$.
\end{corollary}

This consequence is somewhat surprising since inclusion and exclusion atoms can be seen as duals of each other (\cite{Ronnholm15}). As a matter of fact, exclusion atoms can also be simulated with inclusion atoms in an analogous way as we simulated inclusion atoms with exclusion atoms. To see this, suppose that $X$ is a team and $\tuple x$, $\tuple w$, $\tuple w^c$ are tuples variables s.t. $X(\tuple w^c)=\overline{X(\tuple w)}$. Now we have: $\mathcal{M}\true_X \tuple x\exc\tuple w$ iff $\mathcal{M}\true_X\tuple x\inc\tuple w^c$ (c.f. the observation in the beginning of Section~\ref{ssec: Capturing ESO[k]}). 

By the observation above, it would be natural to assume that $\ESOset[k]$-sentences could be expressed with $\INC[k]$-sentences similarly as we did with $\EXC[k]$-sentences. But this is impossible as we deduced above. The problem is that in $\INC$ there is no way to ``force'' the tuples $\tuple w$ and $\tuple w^c$ to be quantified in such a way that their values would be complements of each other.
However, there is a possibility this could be done in inclusion logic with \emph{strict semantics}, since Galliani, Hannula and Kontinen \cite{Kontinen13a} have shown that this logic is equivalent with ESO. 
We will study this question in the next section.


\section{Lower bound for the expressive power of \\ $k$-ary inclusion logic with strict semantics}\label{sec: Expressive power of INCs}

In this section we will study the expressive power of $k$-ary inclusion logic with strict semantics, denoted by INC$^s[k]$. By using similar tricks as in the previous section, we can formulate a translation from ESO[$k$] to INC$^s[k]$ and thus obtain a lower bound for the expressive power of INC$^s[k]$.

In this section we will exclusively use strict-semantics --  both for evaluating existential quantifiers and for evaluating disjunctions. In order to make make this more explicit, we could have chosen to use a different symbol for the truth -- such as $\true^s$. But have we decided keep our notation more simple.


\subsection{Properties of inclusion logic with strict semantics}

Inclusion logic with the alternative (nonequivalent) strict semantics has been studied in e.g. \cite{Kontinen13a} and \cite{Hannula15}.
As we have noted before, when using strict semantics with inclusion logic, we lose the locality property. Hence the resulting logic is a bit strange by having some counterintuitive properties\footnote{Note that IF-logic is not local either. This is manifested by some exotic properties, such as \emph{signaling}.}. We have to be extra careful when formulating our proofs for INC$^s$ since locality is one of the most commonly used properties used in proofs in the framework of team semantics.

Moreover, not only locality of INC is lost with strict semantics. With inclusion logic we very often use its property of being closed under unions. But also this property is lost with strict semantics, as seen by the following example.\footnote{The corresponding observation has been done independently in \cite{Hella17} for \emph{propositional inclusion logic}.}

\begin{example}
The first case below shows that, with strict semantics for disjunction, the closure under unions is lost for INC. The second case shows the same for strict semantics for existential quantifier. For both cases, let $M=\{0,1,2\}$.
\begin{enumerate}
\item Let $\varphi:=x\inc y \vee y\inc x$ and let $X_1=\{s_0,s_1\}$ and $X_2=\{s_0,s_2\}$, where 
\[
	\begin{cases}
		s_0(x)=0 \\
		s_0(y)=0
	\end{cases}
	\qquad
		\begin{cases}
		s_1(x)=0 \\
		s_1(y)=1
	\end{cases}
	\qquad
	\begin{cases}
		s_2(x)=2 \\
		s_2(y)=0.
	\end{cases}
\]
Now $\mathcal{M}\true_{X_1}\varphi$ since we do a trivial splitting of $X_1$ by leaving the right side empty. Similarly $\mathcal{M}\true_{X_2}\varphi$ since we can leave the left side empty when splitting $X_2$. But $\mathcal{M}\ntrue_{X_1\cup X_2}\varphi$ since there is no way to split $X_1\cup X_2=\{s_0,s_1,s_2\}$ into two \emph{disjoint} subteams such that the other would satisfy $x\inc y$ and the other would satisfy $y\inc x$. 
(Note that with lax semantics $X_1\cup X_2$ can here be split into the subteams $\{s_0,s_1\}$ and $\{s_0,s_2\}$.)


\item Let $\varphi:=\Ee z\,(z\neq x\wedge z\neq y\wedge x\inc z)$ and let $X_1=\{s_0,s_1\}$ and $X_2=\{s_0,s_2\}$, where 
\[
	\begin{cases}
		s_0(x)=0 \\
		s_0(y)=0
	\end{cases}
	\qquad
		\begin{cases}
		s_1(x)=1 \\
		s_1(y)=2
	\end{cases}
	\qquad
	\begin{cases}
		s_2(x)=2 \\
		s_2(y)=1.
	\end{cases}
\]
Now $\mathcal{M}\true_{X_1}\varphi$ since we can map $s_0$ to $1$ and $s_1$ to $0$. Similarly $\mathcal{M}\true_{X_2}\varphi$ since we can map $s_0$ to $2$ and $s_2$ to $0$. However, $\mathcal{M}\ntrue_{X_1\cup X_2}\varphi$ since $\abs{(X_1\cup X_2)(x)}=3$, but both $s_1$ and $s_2$ must be mapped to $0$. 
(Note that with lax semantics $s_0$ can  here be mapped to both $1$ and $2$.)
\end{enumerate}
\end{example}


\subsection{Simulating exclusion in $\INC^s$}

In order to formulate a translation from ESO[$k$] to INC$^s[k]$, we need to be able say in INC$^s$ that the exclusion $\vec x_1\exc\vec x_2$ holds for $k$-tuples $\vec x_1$ and $\vec x_2$. In certain cases this is possible; even without access to the complementary values of $\vec x_1$ and $\vec x_2$ in the team.
For this purpose, we consider a variant of \emph{term-value preserving disjunction} (\cite{Ronnholm15}). The disjunction $\varphi\,_{\scriptsize\vec x_1}\!\!\vee_{\scriptsize\vec x_2}\psi$ states the same as normal disjunction, with the additional assumption that the values of $\vec x_1$ are preserved on the left and the values of $\vec x_2$ on the right when the team is split. That is, $\mathcal{M}\true_X\varphi\,_{\scriptsize\vec x_1}\!\!\vee_{\scriptsize\vec x_2}\psi$ holds if and only if there are $Y,Y'\subseteq X$ such that $Y\cup Y'=X$, $Y\cap Y'=\emptyset$, $\mathcal{M}\true_Y\varphi$, $\mathcal{M}\true_{Y'}\psi$ and additionally $Y(\vec x_1)=X(\vec x_1)$ and $Y'(\vec x_2)=X(\vec x_2)$. 

When $\varphi:=\vec x_1\inc\vec z$ and $\psi:=\vec x_2\inc\vec z$, the truth of $\varphi\,_{\scriptsize\vec x_1}\!\!\vee_{\scriptsize\vec x_2}\psi$ (by strict semantics) will guarantee in certain teams that the exclusion $\vec x_1\exc \vec x_2$ holds. Sufficient condition here is that all the values of all variables in $X$ are \emph{dependent} on the values of $\vec z$. When this holds and $X$ is split into disjoint subteams $Y$ and $Y'$, it is then guaranteed that $Y(\vec z)\cap Y'(\vec z)=\emptyset$. Supposing that $\mathcal{M}\true_X\vec x_1\inc\vec z\,_{\scriptsize\vec x_1}\!\!\vee_{\scriptsize\vec x_2}\vec x_2\inc\vec z$, we then have $X(\vec x_1)=Y(\vec x_1)\subseteq Y(\vec z)$ and $X(\vec x_1)=Y'(\vec x_1)\subseteq Y'(\vec z)$, whence it follows that $X(\vec x_1)\cap X(\vec x_2)=\emptyset$.

In Definition~\ref{def: Exclusion for inclusion} the defined operator $\exclusion_{c_l,c_r,\vec z}(\vec x_1, \vec x_2)$ is derived quite directly from the the definition of the disjunction $\vec x_1\inc\vec z\,_{\scriptsize\vec x_1}\!\!\vee_{\scriptsize\vec x_2}\vec x_2\inc\vec z$ in INEX. The definition is very complex, but we try to explain its main idea here briefly. Suppose that $c_l$, $c_r$ have constant values in a team $X$ and that $X(c_l)\neq X(c_r)$.\footnote{With unary dependence atoms $\dep(x)$ we could state the values for these variables in the team are constants. However, since we cannot express these atoms with inclusion atoms, we have to assume this to be the case. (Alternatively we could use some constant symbols which have different interpretations.)} Now we can quantify a ``label variable'' $y$ for each assignment such that it gets either the value of $c_l$ or $c_r$. This value states whether the assignment in question will be placed on the left ($c_l$) or on the right ($c_r$) when evaluating a disjunction that follows this quantification.
Since these label values are given before the team is split, we can ``check'' beforehand by using inclusion atoms that the values of tuple $\vec x_1$ are preserved on the left and the values of $\vec x_2$ are preserved on the right. This is done with formulas $\theta$ and $\theta'$: the truth of $\theta$ guarantees the preservation of all values except for a constant $\vec c_l$ and the truth of $\theta'$ guarantees the preservation for all values except for a constant $\vec c_r$. When $\vec c_l\neq\vec c_r$, the truth of the conjunction $\theta\wedge\theta'$ guarantees the preservation of all values.

\begin{definition}\label{def: Exclusion for inclusion}
Let $c_l$ and $c_r$ be variables and let $\vec z$, $\vec x_1$, $\vec x_2$ be $k$-tuples of variables. We write
\begin{align*}
	&\exclusion_{c_l,c_r,\vec z}(\vec x_1, \vec x_2) :=
	\Ee y\,\big(((y\!=\!c_l\wedge \vec x_1\inc\vec z)\vee(y\!=\!c_r\wedge \vec x_2\inc\vec z ))
	\wedge\theta\wedge\theta'\bigr), \\[3mm]
	&\quad\theta:=\Ee\vec z_1 \Ee\vec z_2
	\bigl(((y \!=\!c_l\wedge\vec z_1\!=\!\vec x_1\wedge\vec z_2\!=\!\vec c_1) \\
	&\quad\hspace{3cm}\vee(y\!=\!c_r\wedge\vec z_1\!=\!\vec c_1\wedge\vec z_2\!=\!\vec x_2))
	\wedge\vec x_1\inc\vec z_1 \wedge \vec x_2\inc\vec z_2\bigr) \\
	&\quad\theta':=\Ee\vec z_1\Ee\vec z_2
	\bigl(((y\!=\!c_l\wedge\vec z_1\!=\!\vec x_1\wedge\vec z_2\!=\!\vec c_2) \\
	&\quad\hspace{3cm}\vee(y\!=\!c_r\wedge\vec z_1\!=\!\vec c_2\wedge\vec z_2\!=\!\vec x_2))
	\wedge\vec x_1\inc\vec z_1\wedge\vec x_2\inc\vec z_2\bigr),
\end{align*}
where $y$ is a fresh variable, $\vec z_1,\vec z_2$ are $k$-tuples of fresh variables and $\vec c_1$ and $\vec c_2$ are $k$-tuples such that $\vec c_1=c_l\dots c_l$  and $\vec c_2=c_r\dots c_r$.
\end{definition}

The following lemma gives sufficient conditions for the truth of $\mathcal{M}\true_X\vec x_1\exc\vec x_2$. This result is needed when proving Theorem~\ref{the: Translation from ESO to INCs} in the next section.

\begin{lemma}\label{the: Exclusion for inclusion}
Let $\mathcal{M}$ be a model and let $X$ be a team, where $c_l$ and $c_r$ have different constant values $a$ and $b$, respectively. Suppose that the $k$-tuples $\vec z$, $\vec x_1$, $\vec x_2$ are all  in $\dom(X)$ and that the variable $y$ in the definition of $\exclusion_{c_l,c_r,\vec z}(\vec x_1, \vec x_2)$ is not in $\dom(X)$. Moreover, assume that the following conditions hold for $X$.
\begin{enumerate}
\item $\mathcal{M}\true_X\dep(\vec z, v)$ for all $v\in\dom(X)$.
\item $\mathcal{M}\true_X \exclusion_{c_l,c_r,\vec z}(\vec x_1, \vec x_2)$.
\end{enumerate}

\medskip
\noindent
Then it holds that $X(\vec x_1)\cap X(\vec x_2)=\emptyset$.
\end{lemma}

\begin{proof}
We write $\vec a:=a\dots a$ and $\vec b:=b\dots b$.
Since $\mathcal{M}\true_X\exclusion_{c_l,c_r,\vec z}(\vec x_1, \vec x_2)$, there is $F: X\rightarrow M$ s.t. $\mathcal{M}\true_{X'}((y\!=\!c_l\wedge \vec x_1\inc\vec z)\vee(y\!=\!c_r\wedge \vec x_2\inc\vec z))\wedge\theta\wedge\theta'$, where $X'=X[F/y]$. Thus $\mathcal{M}\true_{X'}\theta$, $\mathcal{M}\true_{X'}\theta'$ and there are $Y,Y'\subseteq X'$ s.t. $Y\cup Y'=X'$, $Y\cap Y'=\emptyset$, $\mathcal{M}\true_{Y}y\!=\!c_l\wedge \vec x_1\inc\vec z$ and $\mathcal{M}\true_{Y'}y\!=\!c_r\wedge \vec x_2\inc\vec z$. Since $X'(c_l)=\{a\}$ and $X'(c_r)=\{b\}$, it is easy to see that the following conditions hold for any assignment $s\in X'$:
\[
	s\in Y \text{ iff } s(y)=a \qquad \text{ and } \qquad s\in Y' \text{ iff } s(y)=b.
\]

We first show that $Y(\vec z)\cap Y'(\vec z)=\emptyset$. Suppose, for the sake of contradiction, that $Y(\vec z)\cap Y'(\vec z)\neq\emptyset$, whence there is $s\in Y$ and $s'\in Y'$ s.t. $s(\vec z)=s'(\vec z)$. Now $s(y)=a$ and $s'(y)=b$. Since $\mathcal{M}\true_X\dep(\vec z, v)$ for all $v\in\dom(X)$, by the strict semantics of existential quantifier we must have $\mathcal{M}\true_{X'}\dep(\vec z, y)$. But this is impossible since $s(\vec z)=s'(\vec z)$ and $s(y)\neq s'(y)$.

Since $\mathcal{M}\true_{X'}\theta$, there are $\mathcal{F}_1:X'\rightarrow M^k$ and $\mathcal{F}_2:X'[\mathcal{F}_1/\vec z_1]\rightarrow M^k$ s.t. $\mathcal{M}\true_{Z}((y \!=\!c_l\wedge\vec z_1\!=\!\vec x_1\wedge\vec z_2\!=\!\vec c_1)\vee(y\!=\!c_r\wedge\vec z_1\!=\!\vec c_1\wedge\vec z_2\!=\!\vec x_2))\wedge\vec x_1\inc\vec z_1 \wedge \vec x_2\inc\vec z_2$, where $Z:=X[\mathcal{F}_1/\vec z_1,\mathcal{F}_2/\vec z_2]$. Hence $\mathcal{M}\true_{Z}\vec x_1\inc\vec z_1$, $\mathcal{M}\true_{Z}\vec x_2\inc\vec z_2$ and there are $W_1,W_2\subseteq Z$ s.t. $W_1\cup W_2=Z$, $W_1\cap W_2=\emptyset$, $\mathcal{M}\true_{W_1}y \!=\!c_l\wedge\vec z_1\!=\!\vec x_1\wedge\vec z_2\!=\!\vec c_1$ and $\mathcal{M}\true_{W_2}y\!=\!c_r\wedge\vec z_1\!=\!\vec c_1\wedge\vec z_2\!=\!\vec x_2$.

As above, since $\mathcal{M}\true_{X'}\theta'$, there are $\mathcal{F}_1':X'\rightarrow M^k$ and $\mathcal{F}_2':X'[\mathcal{F}_1'/\vec z_1]\rightarrow M^k$ s.t. $\mathcal{M}\true_{Z'}\vec x_1\inc\vec z_1$, $\mathcal{M}\true_{Z'}\vec x_2\inc\vec z_2$, where $Z':=X[\mathcal{F}_1'/\vec z_1,\mathcal{F}_2'/\vec z_2]$. Moreover there are subteams $W_1',W_2'\subseteq Z'$ such that $W_1'\cup W_2'=Z'$, $W_1'\cap W_2'=\emptyset$, $\mathcal{M}\true_{W_1'}y \!=\!c_l\wedge\vec z_1\!=\!\vec x_1\wedge\vec z_2\!=\!\vec c_2$ and $\mathcal{M}\true_{W_2'}y\!=\!c_r\wedge\vec z_1\!=\!\vec c_2\wedge\vec z_2\!=\!\vec x_2$.

Suppose, for the sake of contradiction, that there is $\vec e\in X(\vec x_1)\cap X(\vec x_2)$. Hence there are $s_1,s_2\in X$ s.t. $s_1(\vec x_1)=\vec e=s_2(\vec x_2)$. Let $r_1:=s_1[F(s_1)/y]$ and $r_2:= s_2[F(s_2)/y]$. Now we must have $r_1(y),r_2(y)\in\{a,b\}$. 

Suppose first that $r_1(y)=a$ and $r_2(y)=b$, whence $r_1\in Y$ and $r_2\in Y'$. Since $\mathcal{M}\true_{Y}\vec x_1\inc \vec z$, we must have $\vec e=r_1(\vec x_1)\in Y(\vec z)$. And since $\mathcal{M}\true_{Y'}\vec x_2\inc \vec z$, we must have $\vec e=r_2(\vec x_2)\in Y'(\vec z)$. But this is impossible, since we deduced above that $Y(\vec z)\cap Y'(\vec z)=\emptyset$. The case when $r_1(y)=b$ and $r_2(y)=a$ leads to a contradiction with a symmetric reasoning.

Suppose then that $r_1(y)=a=r_2(y)$, whence $r_1,r_2\in Y$. As above, we must have $\vec e\in Y(\vec z)$.  Let $r_2'\in Z$ be the assignment that is obtained by extending $r_2$ with $\mathcal{F}_1$ and $\mathcal{F}_2$. Since $\mathcal{M}\true_{Z}\vec x_2\inc\vec z_2$, there is $r_3'\in Z$ s.t. $r_3'(\vec z_2)=r_2'(\vec x_2)$. Suppose first that $r_3'\in W_2$, whence $r_3'(y)=r_3'(c_r)$ and $r_3'(\vec z_2)=r_3'(\vec x_2)$. Let $r_3\in X'$ be the assignment that becomes $r_3'$ when extending it with $\mathcal{F}_1$ and $\mathcal{F}_2$. Since $r_3'(y)=r_3'(c_r)$, also $r_3(y)=r_3(c_r)$ and thus $r_3\in Y'$. Now we have $r_3(\vec x_2)=r_3'(\vec x_2)=r_3'(\vec z_2)=r_2'(\vec x_2)=r_2(\vec x_2)=s_2(\vec x_2)=\vec e$ and thus $\vec e\in Y'(\vec x_2)$. But since $\mathcal{M}\true_{Y'}\vec x_2\inc \vec z$, we also have $\vec e\in Y'(\vec z)$. But this is impossible since $\vec e\in Y(\vec z)$ and we have shown that $Y(\vec z)\cap Y'(\vec z)=\emptyset$. Thus $r_3'$ cannot be in $W_2$.

Suppose then that $r_3'\in W_1$, whence $r_3'(\vec z_2)=r_3'(\vec c_1)$ and thus $r_2'(\vec x_2)=\vec a$. Let $r_2''\in Z'$ be the assignment that is obtained by extending $r_2$ with $\mathcal{F}_1' $ and $\mathcal{F}_2'$. Since $\mathcal{M}\true_{Z'}\vec x_2\inc\vec z_2$, there is $r_3''\in Z$ s.t. $r_3''(\vec z_2)=r_2''(\vec x_2)$. If $r_3''\in W_2'$, then we obtain a contradiction with a similar reasoning as for $r_3'$ above. Hence we must have $r_3''\in W_1'$. But then $r_3''(\vec z_2)=r_3'(\vec c_2)=\vec b$. But this a contradiction since $r_3''(\vec z_2)=r_2''(\vec x_2)=r_2(\vec x_2)=r_2'(\vec x_2)=\vec a\neq \vec b$.

The case when $r_1(y)=b=r_2(y)$ leads to a contradiction with a symmetric reasoning to the previous case. Since all the possible cases lead to a contradiction, we must have $X(\vec x_1)\cap X(\vec x_2)=\emptyset$.
\end{proof}

For the proof of Theorem~\ref{the: Translation from ESO to INCs}, we need also some sufficient conditions for the truth of $\exclusion_{c_l,c_r,\vec z}(\vec x_1, \vec x_2)$ in a team. The assumptions in the following lemma are very specific as this lemma is formulated particularly for the proof of Theorem~\ref{the: Translation from ESO to INCs}.
(The operator $\exclusion_{c_l,c_r,\vec z}(\vec x_1, \vec x_2)$ is not very interesting in its own right -- it is just a tool for our translation from ESO[$k$] to INC$^s$[$k$].)

\begin{lemma}\label{the: Exclusion for inclusion 2}
Let $\mathcal{M}$ be a model and let $X$ be a team where $c_l$ and $c_r$ have different constant values $a$ and $b$, respectively. We assume that the following conditions hold for the team $X$.
\begin{enumerate}
\item $X(\vec x_1)\cap X(\vec x_2)=\emptyset$.
\item For each $s\in X$ either $s(\tuple z\,)=s(\tuple x_1)$ or $s(\tuple z\,)=s(\tuple x_2)$.
\item For each $\vec a_1\in X(\vec x_1)$, there is $s\in X$ for which $s(\tuple x_1)=\vec a_1=s(\vec z\,)$.
\item For each $\vec a_2\in X(\vec x_2)$, there is $s\in X$ for which $s(\tuple x_2)=\vec a_2=s(\vec z\,)$.
\end{enumerate}

\medskip

\noindent
Now it holds that $\mathcal{M}\true_X \exclusion_{c_l,c_r,\vec z}(\vec x_1, \vec x_2)$.
\end{lemma}

\begin{proof}
We first note that by the assumptions 1 and 2, it is impossible that $s(\vec x_1)=s(\vec z)=s(\vec x_2)$ for any $s\in X$. Hence, by the assumption 2, we can define the following function: 
\[
	F:X\rightarrow M, \quad
	\begin{cases}
		s\mapsto a \quad\, \text{ if } s(\vec z\,)=s(\vec x_1) \\
		s\mapsto b \quad\; \text{ if } s(\vec z\,)=s(\vec x_2).
	\end{cases}
\]
Let $X':=X[F/y]$, $Y:=\{s\in X\mid s(y)=a\}$ and $Y:=\{s\in X\mid s(y)=b\}$. Now clearly $Y,Y'\subseteq X'$, $Y\cup Y'=X'$ and $Y\cap Y'=\emptyset$. By the definition of $F$ it is easy to see that $\mathcal{M}\true_{Y}y\!=\!c_l\wedge \vec x_1\inc\vec z$ and $\mathcal{M}\true_{Y'}y\!=\!c_r\wedge \vec x_2\inc\vec z$. Let $\vec a$ be the $k$-tuple $\vec a:= a\dots a$ and let
\begin{align*}
	&\mathcal{F}_1:X'\rightarrow M^k, \quad
	\begin{cases}
		s\mapsto s(\vec x_1) \;\; \text{ if } s(y)=a \\
		s\mapsto \vec a \qquad\; \text{ if } s(y)=b
	\end{cases} \\
	&\mathcal{F}_2:X'[\mathcal{F}_1/\vec z_1]\rightarrow M^k, \quad
	\begin{cases}
		s\mapsto \vec a \qquad\; \text{ if } s(y)=a \\
		s\mapsto s(\vec x_2) \;\; \text{ if } s(y)=b.
	\end{cases}
\end{align*}

We define the teams $Z:=X[\mathcal{F}_1/\vec z_1,\mathcal{F}_2/\vec z_2]$, $W_1:=\{s\in Z\mid s(y)=a\}$ and $W_2:=\{s\in Z\mid s(y)=b\}$. We clearly have $W_1\cup W_2=Z$, $W_1\cap W_2=\emptyset$, $\mathcal{M}\true_{W_1}y \!=\!c_l\wedge\vec z_1\!=\!\vec x_1\wedge\vec z_2\!=\!\vec c_1$ and $\mathcal{M}\true_{W_2}y\!=\!c_r\wedge\vec z_1\!=\!\vec c_1\wedge\vec z_2\!=\!\vec x_2$.
For the sake of showing that $\mathcal{M}\true_{Z}\vec x_1\inc\vec z_1$, let $r\in Z$. Let $s\in X$ be the assignment that becomes $r$, when it is extended with $F$, $\mathcal{F}_1$ and $\mathcal{F}_2$. By the assumption 3, there is $s'\in X$ such that $s'(\vec z)=s(\vec x_1)=s'(\vec x_1)$. Let then $r':=s'[a/y,s'(\vec x_1)/\vec z_1,\vec a/\vec z_2]$. Now $r'\in W_1$ and $r'(\vec z_1)=s'(\vec x_1)=s(\vec x_1)=r(\vec x_1)$.
By using the assumption 4, we can analogously show that $\mathcal{M}\true_{Z}\vec x_2\inc\vec z_2$ and therefore $\mathcal{M}\true_{X'}\theta$. Moreover, we can show by a similar reasoning that $\mathcal{M}\true_{X'}\theta'$, which concludes the proof.
\end{proof}


\subsection{Translation from $\ESO[k]$ to $\INC^s[k]$}

We can formulate a translation from ESO[$k$] to INC$^s$[$k$] by using very similar ideas as in our translation form ESO[$k$] to EXC[$k$]. As noticed before, we can simulate exclusion atoms with inclusion atoms if we have access to the complementary values in the team: Let $X$ be a team and $\tuple x$, $\tuple w_i$, $\tuple w_i^c$ tuples s.t. $X(\tuple w_i^c)=\overline{X(\tuple w_i)}$. Then we have: $\mathcal{M}\true_X \tuple x\exc\tuple w_i$ iff $\mathcal{M}\true_X\tuple x\inc\tuple w_i^c$.

As in the translation in the proof of Theorem~\ref{the: Translation from ESO to EXC}, we use label variables $w_i^\circ$ and $w_i^\bullet$ for simulating the quantification of the empty relation and the full relation $M^k$. Furthermore, we need again Lemma~\ref{the: Single element} for handling the special case of single element models.
One surprising feature of this translation is that we can translate disjunctions directly as $(\psi\vee\theta)' = \psi'\vee\theta'$; this time there is no need for term value preserving disjunction or any other trick as we may allow some of the values of tuples $\tuple w_i$, $\tuple w_i^c$ to be lost when evaluating disjunctions.

The structure of the following proof has many similarities with the proof of  Theorem~\ref{the: Translation from ESO to EXC} and we will omit the parts that can be done here analogously. However, there are also many parts that look similar but which are proven by using different assumptions and thus need to be presented with all the details. 

\begin{theorem}\label{the: Translation from ESO to INCs}
Let $\Phi$ be an $\ESOset[k]$-sentence. Now there exists an $\INCset[k]$-sentence $\varphi$ such that
\[
	\mathcal{M}\true\varphi \; \text{ iff } \; \mathcal{M}\true\Phi.
\]
\end{theorem}

\begin{proof}
Since $\Phi$ is an $\ESOset[k]$-sentence, there exists a $\FOset$-sentence $\delta$ and relation symbols $P_1,\dots,P_n$ so that $\Phi=\Ee P_1\dots\Ee P_n\delta$. We may assume again that $P_1,\dots,P_n$ are all $k$-ary. Let $w_1^\circ,\dots,w_n^\circ,w_1^\bullet,\dots,w_n^\bullet$, $\tuple w_1,\dots,\tuple w_n$ and $\tuple w_1^c,\dots,\tuple w_n^c$ be as in the proof of Theorem~\ref{the: Translation from ESO to EXC}. Let $u$ and $u'$ be fresh variables.
\\ \\
Let $\psi\in\subf(\delta)$. The formula $\psi'$ is defined recursively:
\begin{align*}
	\psi' &= \psi, \text{ if $\psi$ is a literal and $P_i$ does not occur in } \psi \text{ for any } i\leq n \\
	(P_i \vec t)' &= (\vec t\inc\vec w_i\vee w_i^\bullet\!=\!u) \wedge w_i^\circ\!\neq\!u \quad\text{ for all } i\leq n \\
	(\neg P_i \vec t)' &= (\vec t\inc\vec w_i^c\vee w_i^\circ\!=\!u) \wedge w_i^\bullet\!\neq\!u \quad\text{ for all } i\leq n \\
	(\psi\wedge\theta)' &= \psi'\wedge\theta' \\
	(\psi\vee\theta)' &= \psi'\vee\theta' \\
	(\Ee x\,\psi)' &= \Ee x\,\psi' \\
	(\Ae x\,\psi)' & = \Ae x\,\psi'.
\end{align*}
Let $\chi$ be a $\FOset$-sentence determined by the lemma \ref{the: Single element} for the sentence $\Phi$ and let $\vec z$ be a $k$-tuple of fresh variables. 
We can now define $\varphi$ as follows:
\begin{align*}
	\varphi &:= (\gamma_{=1}\wedge\chi)\vee\Ee u\Ee u'\,\Bigl(u\neq u'
	\wedge\Ee w_1^\circ\dots\Ee w_n^\circ\Ee w_1^\bullet\dots\Ee w_n^\bullet\\[-0,1cm]
	&\hspace{3,5cm}\Ae\vec{z}\Ee\vec w_1\dots\Ee\vec w_n
	\Ee\vec w_1^c\dots\Ee\vec w_n^c
	\bigl(\bigwedge_{i=1}^n\exclusion_{u,u',\vec z}(\vec w_i,\vec w_i^c)
	\wedge\delta'\bigr)\Bigr).
\end{align*}
Clearly now $\varphi$ is an $\INCset[k]$-sentence.

\begin{remark}
Since we are using the tuples $\vec w_i$ and $\vec w_i^c$ to simulate a quantified relation and its complement, it would be natural to require that the union of these values forms the full relation $M^k$. This could be achieved by adding the requirement $\bigwedge_{i\leq n}\forall\,\vec v\,(\vec v\inc w_i\vee \vec v\inc\vec w_i^c)$ to the sentence $\varphi$ above. However, we will see that this is not necessary, since it suffices that $\vec w_i$ and $\vec w_i^c$ are quantified in such a way that $X(\tuple w_i)\cap X(\tuple w_i^c)=\emptyset$ in the resulting team.\footnote{Recall that in the proof of Theorem~\ref{the: Translation from ESO to EXC} we had to require that $X(\tuple w_i)\cup X(\tuple w_i^c)=M^k$. This difference forms is an interesting piece of duality between these two translations.}
\end{remark}

Before proving the claim of this theorem, we prove the following two claims. The first claim is quite similar to Claim~\ref{R1}. But here instead of assuming that $X(\tuple w_i)\cup X(\tuple w_i^c)=M^k$ we dually assume that $X(\tuple w_i)\cap X(\tuple w_i^c)=\emptyset$. Also, when defining the sets $A_i$, we cannot simply define $A_i=X(\vec w_i)$ as before. Instead, we must prove that \emph{any set} $B$ for which $X(\tuple w_i)\subseteq B\subseteq \overline{X(\tuple w_i^c)}$ could be chosen as $A_i$ (this requirement makes sense since $X(\tuple w_i)\cap X(\tuple w_i^c)=\emptyset$ for each $i\leq n$). This strengthening of the claim is crucial for proving the case of disjunction.

\medskip

\noindent
We write 
\[
V^* := \vr(u u' w_1^\circ \dots w_n^\circ w_1^\bullet\dots w_n^\bullet\tuple w_1\dots\tuple w_n\tuple w_1^c\dots\tuple w_n^c).
\]

\begin{claim}\label{R3}
Let $\mathcal{M}$ be an $L$-model with at least two elements. Let $\mu\in\subf(\delta)$ and let $X$ a team for which $V^*\!\subseteq\!\dom(X)$ and the following assumptions hold:
\[
	\begin{cases}
		X(\tuple w_i)\cap X(\tuple w_i^c)=\emptyset \text{ for each } i\leq n. \\
		\text{The values of } w_i^\circ,w_i^\bullet \;(i\leq n), u \text{ and } u' \text{ are constants in } X.
	\end{cases}
\]
We consider functions $H_X:\{1,\dots,n\}\rightarrow\mathcal{P}(M^k)$ s.t. for each $i\leq n$ we have
\[
	X(\tuple w_i)\subseteq H_X(i)\subseteq \overline{X(\tuple w_i^c)}.
\]
Let $\mathcal{M}_{H_X} := \mathcal{M}[\tuple A/\tuple P]$, where
\begin{align*}
	A_i &=
	\begin{cases}
		\emptyset \qquad\,\text{ if } X(w_i^\circ)=X(u) \text{ and }\, X(w_i^\bullet)\neq X(u) \\
		M^k \quad\,\text{ if } X(w_i^\bullet)=X(u) \text{ and }\, X(w_i^\circ)\neq X(u) \\
		H_X(i) \text{ else}.
	\end{cases}
\end{align*}

\noindent
Now the following implication holds for every function $H_{X}$: 
\[
	\text{If } \mathcal{M}\true_X\mu', \text{ then } \mathcal{M}_{H_X}\true_X\mu.
\]
\end{claim}

\noindent
We prove this claim by structural induction on $\mu$:
\begin{itemize}
\item If $\mu$ is a literal and $P_i$ does not occur in $\mu$ for any $i\leq n$, then the claim holds trivially since $\mu'=\mu$.

\smallskip

\item Let $\mu=P_j\tuple t$ for some $j\leq n$. 
Suppose that we have $\mathcal{M}\true_X (P_j\tuple t\,)'$, i.e. $\mathcal{M}\true_X (\vec t\inc\vec w_j\vee w_j^\bullet\!=\!u) \wedge w_j^\circ\!\neq\!u$. Because the values of $u$, $w_j^\circ$ are constants in $X$ and $\mathcal{M}\true_X w_j^\circ\!\neq\!u$, we have $X(w_j^\circ)\neq X(u)$. 
If $X(w_j^\bullet)=X(u)$, then $A_j=M^k$ and thus trivially $\mathcal{M}_{H_X}\true_X P_j\tuple t$. 
Suppose then that $X(w_j^\bullet)\neq X(u)$ whence $A_j=H_X(j)$. Because the values of  $u$, $w_j^\bullet$ are constants in $X$ and $\mathcal{M}\true_X\vec t\inc\vec w_j\vee w_j^\bullet\!=\!u$, it must hold that $\mathcal{M}\true_X\vec t\inc\vec w_j$. Now $X(\tuple t\,)\subseteq X(\tuple w_j)\subseteq H_X(j)=A_j$ and therefore $\mathcal{M}_{H_X}\true_X P_j\tuple t$.


\item Let $\mu=\neg P_j\tuple t$ for some $j\leq n$.
Suppose that we have $\mathcal{M}\true_X (\neg P_j\tuple t\,)'$, i.e. $\mathcal{M}\true_X(\vec t\inc\vec w_j^c\vee w_j^\circ\!=\!u) \wedge w_j^\bullet\!\neq\!u$. Because the values of $u$, $w_j^\bullet$ are constants and $\mathcal{M}\true_X w_j^\bullet\!\neq\!u$, we have $X(w_j^\bullet)\neq X(u)$. 
If $X(w_j^\circ)=X(u)$, then $A_j=\emptyset$ and thus trivially $\mathcal{M}_{H_X}\true_X \neg P_j\tuple t$. 
Suppose then that $X(w_i^\circ)\neq X(u)$ whence $A_j=H_X(j)$. Because the values of $u$, $w_j^\circ$ are constants in $X$ and $\mathcal{M}\true_X\vec t\inc\vec w_j^c\vee w_j^\circ\!=\!u$, we have $\mathcal{M}\true_X\vec t\inc\vec w_j^c$. Because $H_X(j)\subseteq\overline{X(w_j^c)}$, it also holds that $X(\tuple w_j^c)\subseteq\overline{H_X(j)}$. Therefore $X(\tuple t\,)\subseteq X(\tuple w_j^c)\subseteq\overline{H_X(j)}=\overline{A_j}$ and thus $\mathcal{M}_{H_X}\true_X \neg P_j\tuple t$.


\item Let $\mu=\psi\vee\theta$. 
Suppose that $\mathcal{M}\true_X(\psi\vee\theta)'$, i.e. $\mathcal{M}\true_X\psi'\vee\theta'$. Hence there are $Y,Y'\subseteq X$ s.t. $Y\cup Y'=X$, $Y\cap Y'=\emptyset$,  $\mathcal{M}\true_{Y}\psi'$ and $\mathcal{M}\true_{Y'}\theta'$. Since $X(\tuple w_i)\cap X(\tuple w_i^c)=\emptyset$ for each $i\leq n$, we must also have $Y(\tuple w_i)\cap Y(\tuple w_i^c)=\emptyset=Y'(\tuple w_i)\cap Y'(\tuple w_i^c)$ for each $i\leq n$. Moreover, since the values of  $w_i^\circ$, $w_i^\bullet$ ($i\leq n$) and $u$ are constants in $X$, they must also have (the same) constant values in $Y$ and $Y'$. 

By the inductive hypothesis, $\mathcal{M}_{H_Y}\true_{Y}\psi$ and $\mathcal{M}_{H_{Y'}}\true_{Y}\theta$, for every function $H_{Y}$ and $H_{Y'}$. We then consider an arbitrary function $H_X$. Since $Y(\vec w_i)\subseteq X(\vec w_i)$ and $\overline{ X(\vec w_i^c)}\subseteq\overline{Y(\vec w_i^c)}$ for each $i\leq n$, we have $\mathcal{M}_{H_X}\true_{Y}\psi$. By a symmetric argumentation $\mathcal{M}_{H_X}\true_{Y'}\theta$. Therefore $\mathcal{M}_{H_X}\true_{X}\psi\vee\theta$.


\item The cases $\mu=\psi\wedge\theta$, $\mu=\Ee x\,\psi$ and $\mu=\Ae x\,\psi$ are straightforward to prove.
\end{itemize}

\medskip

The next claim is very similar to Claim~\ref{R2}. However, since we cannot use locality nor downward closure properties with INC$^s$, we must prove this claim more generally for an extended team $Z$ which: (1) matches with $X$ when restricted $\dom(X)$; and (2) has the same values as $X'$ for certain tuples in $V^*$.

\begin{claim}\label{R4}
Let $\mathcal{M}$ be an $L$-model with at least two elements. Let $\mu\in\subf(\delta)$ and $X$ be a team such that $\dom(X)=\fr(\mu)$. Assume that $A_1,\dots,A_n\subseteq M^k$, $\mathcal{M}' := \mathcal{M}[\tuple A/\tuple P\,]$ and $a,b\in M$ s.t. $a\neq b$. 
We write $\vec a:=a\dots a$. Let now
\begin{align*}
	X' &:= X\bigl[\{a\}/u,\{b\}/u',B_1^\circ/w_1^\circ,\dots,B_n^\circ/w_n^\circ,
	B_1^\bullet/w_1^\bullet,\dots,B_n^\bullet/w_n^\bullet, \\[-0,1cm]
	&\hspace{3,8cm}B_1/\tuple w_1,\dots,B_n/\tuple w_n,B_1^c/\tuple w_1^c,\dots,B_n^c/\tuple w_n^c\bigr],\\
	&\text{ where }\,
	\begin{cases}
		B_i^\circ = \{a\},\; B_i^\bullet = \{b\}, B_i = \{\vec a\} \text{ and }  B_i^c =M^k\setminus\{\vec a\}
		\quad \text{ if } A_i=\emptyset \\
		B_i^\circ = \{b\},\; B_i^\bullet = \{a\}, B_i = \{\vec a\} \text{ and }  B_i^c =M^k\setminus\{\vec a\}
		 \quad \text{ if } A_i=M^k \\
		B_i^\circ = \{b\},\; B_i^\bullet = \{b\}, \; B_i = A_i \text{ and } B_i^c = \overline{A_i} \hspace{17,7mm} \text{ else}.
	\end{cases}
\end{align*}
Now the following implication holds:
\[
	\text{If } \mathcal{M}'\true_{X}\mu, 
	\;\text{ then } \mathcal{M}\true_{Z}\mu',
\]
for any team $Z$ for which $Z\upharpoonright\dom(X)=X$
and $Z(\vec v)=X'(\vec v)$ for all $\vec v\in \vec V^*$, where
\[
	\vec V^* := \{u,u'\}\cup\bigcup_{i\leq n}\{w_i^\circ,w_i^\bullet,\vec w_i, \vec w_i^c\}.
\]
\end{claim}
\noindent
We prove this claim by structural induction on $\mu$. If $X=\emptyset$, then also $Z=\emptyset$ and thus the claim holds trivially. Hence we may assume that $X\neq\emptyset$.
\medskip
\begin{itemize}
\item If $\mu$ is a literal and $P_i$ does not occur in $\mu$ for any $i\leq n$, then the claim holds by locality (since literals are first order, we may use locality here).


\item Let $\mu=P_j\tuple t$ for some $j\leq n$. 
Suppose that we have $\mathcal{M}'\true_X P_j\tuple t$, i.e. $X(\tuple t\,)\subseteq P_j^\mathcal{M'}=A_j$. Since $X\neq\emptyset$, also $X(\tuple t\,)\neq\emptyset$ and thus $A_j\neq\emptyset$. Hence $Z(w_j^\circ)=X'(w_j^\circ)=\{b\}$. Since $Z(u)=X'(u)=\{a\}$, we have $\mathcal{M}\true_{Z}w_j^\circ\!\neq\!u$. 
If $A_j=M^k$, then $Z(w_j^\bullet)=X'(w_j^\bullet)=\{a\}$ and thus $\mathcal{M}\true_{Z}w_j^\bullet=u$. Then $\mathcal{M}\true_{Z}(\vec t\subseteq\vec w_j\vee w_j^\bullet\!=\!u) \wedge w_j^\circ\neq u$, i.e. $\mathcal{M}\true_{Z}(P_j\tuple t\,)'$.
Suppose then that $A_j\neq M^k$. Now $Z(\tuple w_j)=X'(\tuple w_j)=A_j$ and thus $Z(\tuple t\,)\!=\!X(\tuple t\,)\subseteq A_j\!=\!Z(\tuple w_j)$. Hence $\mathcal{M}\true_{Z}\vec t\subseteq\vec w_j $ and therefore we have $\mathcal{M}\true_{Z}(\vec t\subseteq\vec w_j\vee w_j^\bullet\!=\!u) \wedge w_j^\circ\neq u$, i.e. $\mathcal{M}\true_{Z}(P_j\tuple t\,)'$.


\item Let $\mu=\neg P_j\tuple t$ for some $j\leq n$.
Suppose that we have $\mathcal{M}'\true_X\neg P_j\tuple t$, i.e. $X(\tuple t\,)\subseteq \overline{P_j^\mathcal{M'}}=\overline{A_j}$. Since $X\neq\emptyset$, also $X(\tuple t\,)\neq\emptyset$ and thus $\overline{A_j}\neq\emptyset$, i.e. $A_j\neq M^k$.  Hence $Z(w_j^\bullet)=X'(w_j^\bullet)=\{b\}$. Since $Z(u)=X'(u)=\{a\}$, we have $\mathcal{M}\true_{Z}w_i^\bullet\!\neq\!u$. 
If $A_j=\emptyset$, then $Z(w_j)=X'(w_j^\circ)=\{a\}$ and thus $\mathcal{M}\true_{X'}w_j^\circ=u$, whence $\mathcal{M}\true_{X'}(\tuple t\inc\tuple w_j^c\vee w_j^\circ\!=\!u)\wedge w_j^\bullet\!\neq\!u$, i.e. $\mathcal{M}\true_{Z}(\neg P_j\tuple t\,)'$.
Suppose then that $A_j\neq\emptyset$. Now $Z(\tuple w_j^c)=X'(\tuple w_j^c)=\overline{A_j}$ and thus $Z(\tuple t\,)\!=\!X(\tuple t\,)\subseteq \overline{A_j}\!=\!Z(\vec w_j^c)$. 
Hence $\mathcal{M}\true_{X'} \tuple t\inc\!\tuple w_j^c$ and thus $\mathcal{M}\true_{X'}(\tuple t\inc\tuple w_j^c\vee w_j^\circ\!=\!u)\wedge w_j^\bullet\!\neq\!u$, i.e. $\mathcal{M}\true_{X'}(\neg P_j\tuple t\,)'$.

\smallskip

\item The case $\mu=\psi\wedge\theta$ is straightforward to prove.


\item Let $\mu=\psi\vee\theta$. 
Suppose that $\mathcal{M}'\true_X\psi\vee\theta$, i.e. there are $Y_1,Y_2\subseteq X$ s.t. $Y_1\cup Y_2=X$, $Y_1\cap Y_2=\emptyset$, $\mathcal{M}'\true_{Y_1}\psi$ and $\mathcal{M}'\true_{Y_2}\theta$. Let $Y_1',Y_2'$ be the teams obtained by extending the teams $Y_1,Y_2$ as $X'$ is obtained by extending $X$. We define the following teams $W_1,W_2\subseteq Z$:
\[
	\begin{cases}
		W_1 := \{s\in Z\mid s\upharpoonright\dom(X)\in Y_1\} \\
		W_2 := \{s\in Z\mid s\upharpoonright\dom(X)\in Y_2\}.
	\end{cases}
\]
Now $W_1\upharpoonright\dom(Y_1)=W_1\upharpoonright\dom(X)=Y_1$ and $W_1(\vec v)=Y_1'(\vec v)$ for all $\vec v\in \vec V^*$. Thus, by the inductive hypothesis, $\mathcal{M}\true_{W_1}\psi'$. By similar reasoning $\mathcal{M}\true_{W_2}\theta'$. It is also easy to see that $W_1\cup W_2=Z$ and $W_1\cap W_2=\emptyset$, whence $\mathcal{M}\true_{Z}\psi'\vee\theta'$, i.e. $\mathcal{M}\true_{Z}(\psi\vee\theta)'$.


\item Let $\mu=\Ee x\,\psi$ (the case $\mu=\Ae x\,\psi$ is proven by a similar reasoning).
Suppose $\mathcal{M}'\true_X\Ee x\,\psi$, i.e. there is $F:X\rightarrow M$ s.t. $\mathcal{M}'\true_{W}\psi$, where $W:=X[F/x]$. Let
\begin{align*}
	G:Z\rightarrow M, \;\;\; s\mapsto F(s\!\upharpoonright\!\fr(\mu)).
\end{align*}
Note that $G$ is well defined since $\dom(X)=\fr(\mu)$ and $Z\upharpoonright\dom(X)=X$.

Let $W'$ be a team that is obtained by extending the team $W$ analogously as $X'$ is obtained by extending $X$. Now by the definition of $G$ we observe that $Z[G/x]\upharpoonright\dom(W)=W$. Moreover, $(Z[G/x])(\vec v)=W'(\vec v)$ for all $\vec v\in\vec V^*$.
Hence, by the inductive hypothesis, $\mathcal{M}\true_{Z[G/x]}\psi'$. Therefore $\mathcal{M}\true_{Z}\Ee x\,\psi'$, i.e. $\mathcal{M}\true_{Z}(\Ee x\,\psi)'$.
\end{itemize}

\medskip

We are now ready to prove the claim of this theorem:
\[
	\mathcal{M}\true\varphi \; \text{ iff } \; \mathcal{M}\true\Phi.
\]
Suppose first that $\mathcal{M}\true\varphi$. Since the standard disjunction $\vee$ is equivalent with intuitionistic disjunction $\sqcup$  for the singleton team $\{\emptyset\}$, either $\mathcal{M}\true\gamma_{=1}\wedge\chi$ or
\begin{align*}
	\mathcal{M}\true &\Ee u\Ee u'\,\Bigl(u\neq u'
	\wedge\Ee w_1^\circ\dots\Ee w_n^\circ\Ee w_1^\bullet\dots\Ee w_n^\bullet \tag{$\star\star$} \\[-0,1cm]
	&\hspace{2cm}\Ae\vec{z}\Ee\vec w_1\dots\Ee\vec w_n 
	\Ee\vec w_1^c\dots\Ee\vec w_n^c 
	\bigl(\bigwedge_{i=1}^n\exclusion_{u,u',\vec z}(\vec w_i,\vec w_i^c)
	\wedge\delta'\bigr)\Bigr).
\end{align*}

If $\mathcal{M}\true\gamma_{=1}\wedge\chi$, the claim holds by Lemma \ref{the: Single element}. Suppose then ($\star\star$), whence by a similar reasoning as in the proof of Theorem~\ref{the: Translation from ESO to EXC}, there exists a team $X_4$ such that the following conditions hold\footnote{The team $X_4$ here matches the team $X_4$ in corresponding part of the proof of Theorem~\ref{the: Translation from ESO to EXC} with the addition that the variable $u'$ is quantified here as a constant that is different from the value of $u$.}:
\begin{itemize}
\item The values of the variables $u$, $u'$, $w_i^\circ$, $w_i^\bullet$ ($i\leq n$) are constants in $X_4$ and moreover $X_4(u)\neq X_4(u')$.
\item $\mathcal{M}\true_{X_4}\bigwedge_{i=1}^n\exclusion_{u,u',\vec z}(\vec w_i,\vec w_i^c)\wedge\delta'$.
\end{itemize}
We first note that since the variables in $\vec z$ were universally quantified and all the other variables in $\dom(X_4)$ were existentially quantified (by the strict semantics), it holds that $\mathcal{M}\true_{X_4}\dep(\vec z, v)$ for all $v\in\dom(X)$.

Let $j\leq n$. Now the assumptions of Lemma~\ref{the: Exclusion for inclusion} hold when $X=X_4$, $c_l=u$, $c_r=u'$, $\vec x_1=\vec w_j$ and $\vec x_2=\vec w_j^c$. Hence by Lemma~\ref{the: Exclusion for inclusion} $X_4(\tuple w_j)\cap X_4(\tuple w_j^c)=\emptyset$. 
Now all the assumptions of Claim \ref{R3} hold for the team $X_4$. Let $\mathcal{M}' := \mathcal{M}[\tuple A/\tuple P]$, where
\begin{align*}
	A_i &=
	\begin{cases}
		\emptyset \qquad\;\;\text{ if } X_4(w_i^\circ)=X_4(u) \text{ and }\, X_4(w_i^\bullet)\neq X_4(u) \\
		M^k \quad\;\;\text{ if } X_4(w_i^\bullet)=X_4(u) \text{ and }\, X_4(w_i^\circ)\neq X_4(u) \\
		X_4(\tuple w_i) \text{ else}.
	\end{cases}
\end{align*}
Since $\mathcal{M}\true_{X_4}\delta'$, by Claim \ref{R3} we have $\mathcal{M'}\true_{X_4}\delta$ (note that any $A_i$ for which $X(\vec w_i)\subseteq A_i\subseteq\overline{X(\vec w_i^c)}$, could have been chosen in the last case above). By locality\footnote{Note that $\delta$ here is an $\FOset$-sentence and thus locality property may be used.} $\mathcal{M'}\true\delta$, and therefore $\mathcal{M}\true\Phi$.

\bigskip

Suppose then that $\mathcal{M}\true\Phi$. If $\abs{M}=1$, then by Lemma \ref{the: Single element} we have $\mathcal{M}\true\gamma_{=1}\wedge\chi$ and thus $\mathcal{M}\true\varphi$. Hence we may assume $\abs{M}\geq 2$, whence there are $a,b\in M$ s.t. $a\neq b$. Since $\mathcal{M}\true\Phi$, there exist $A_1,\dots,A_n\subseteq M^k$ s.t. $\mathcal{M}[\tuple A/\tuple P]\true\delta$. Let
\begin{align*}
	X' := \{\emptyset\}\bigl[\{a\}/u,\{b\}/u',\,&B_1^\circ/w_1^\circ,\dots,B_n^\circ/w_n^\circ,
	B_1^\bullet/w_1^\bullet,\dots,B_n^\bullet/w_n^\bullet, \\
	&B_1/\tuple w_1,\dots,B_n/\tuple w_n,B_1^c/\tuple w_1^c,\dots,B_n^c/\tuple w_n^c\bigr],
\end{align*}
where $B_i^\circ, B_i^\bullet, B_i, B_i^c$ $(i\leq n)$ are defined as in the assumptions of Claim \ref{R4}. 

Let $\mathcal{F}:\{\emptyset\}\rightarrow M^{2n+2}$ be the function that gives value $a$ for $u$, value $b$ for $u'$ and values for variables $w_i^\bullet$, $w_i^\circ$ $(i\leq n)$ exactly as the corresponding function $\mathcal{F}$ in the proof of Theorem~\ref{the: Translation from ESO to EXC}.
%
%
Let now $X_1:=\{\emptyset\}[\mathcal{F}/uu' w_1^\circ\dots w_n^\circ w_1^\bullet\dots w_n^\bullet]$ and $X_2:=X_1[M^k/\tuple z\,]$.
We write $\vec a:=a\dots a$ and fix some $\tuple b_i\in A_i$ for each $i\leq n$ for which $A_i\neq\emptyset$. We define then the following functions
\begin{align*}
	\mathcal{F}_i:X_2[\mathcal{F}_1/\tuple w_1,\dots,\mathcal{F}_{i-1}/\tuple w_{i-1}]\rightarrow M^k,
	\begin{cases}
		s\mapsto \tuple a \;\,\quad\text{ if } A_i=\emptyset \text{ or } A_i=M^k \\
		s\mapsto s(\tuple z) \:\text{ if } s(\tuple z)\in A_i \text{ and } A_i\neq M^k \\
		s\mapsto \tuple b_i \;\quad\text{ if } s(\vec z)\notin A_i \text{ and } A_i\neq\emptyset.
	\end{cases} 
\end{align*}
Let $X_3:=X_2[\mathcal{F}_1/\tuple w_1,\dots,\mathcal{F}_n/\tuple w_n]$. We write $\vec b:=b\dots b$ and fix some $\tuple b_i'\in\overline{A_i}$ for each $i\leq n$ for which $A_i\neq M^k$. Let
\begin{align*}
	\mathcal{F}_i':X_3[\mathcal{F}_1'/\tuple w_1^c,\dots,\mathcal{F}_{i-1}'/\tuple w_{i-1}^c]\rightarrow M^k,
	\begin{cases}
		s\mapsto s(\tuple z) \!\:\text{ if } A_i\in\{\emptyset,M^k \} \text{ and } s(\vec z)\neq\vec a \\
		s\mapsto \vec b \!\;\;\quad\text{ if } A_i\in\{\emptyset,M^k \} \text{ and } s(\vec z)=\vec a \\
		s\mapsto s(\tuple z) \!\:\text{ if } s(\tuple z)\notin A_i \text{ and } A_i\neq\emptyset \\
		s\mapsto \tuple b_i' \!\;\quad\text{ if } s(\vec z)\in A_i \text{ and } A_i\neq M^k.
	\end{cases} 
\end{align*}
Let $X_4:=X_3[\mathcal{F}_1'/\tuple w_1^c,\dots,\mathcal{F}_n'/\tuple w_n^c]$.

Let $j\le n$. By observing the definitions of $\mathcal{F}_j'$ and $\mathcal{F}_j'$, we can see that all the assumptions of Lemma~\ref{the: Exclusion for inclusion 2} hold when $X=X_4$, $c_l=u$, $c_r=u'$, $\vec x_1=\vec w_j$ and $\vec x_2=\vec w_j^c$. Hence by Lemma~\ref{the: Exclusion for inclusion 2} we have $\mathcal{M}\true_{X_4}\exclusion_{u,u',\vec z}(\vec w_j,\vec w_j^c)$.
Moreover, it is easy to see that $X_4(\vec v)=X'(\vec v)$ for every $\vec v\in\vec V^*$ (recall the assumptions of Claim~\ref{R4}). Since $\mathcal{M}[\tuple A/\tuple P]\true\delta$ and $X_4\upharpoonright\dom(\{\emptyset\})=\{\emptyset\}$ by Claim~\ref{R4} we have $\mathcal{M}\true_{X_4}\delta'$
(note that we cannot use locality property nor downwards closure here as in the proof of Theorem~\ref{the: Translation from ESO to EXC}).
Therefore $\mathcal{M}\true_{X_4}\bigwedge_{i\leq n}\exclusion_{u,u',\vec z}(\vec w_i,\vec w_i^c)\wedge\delta'$ and moreover $\mathcal{M}\true\varphi$.
\end{proof}

\begin{corollary}\label{the: Upper bound for INC^s}
On the level of sentences \emph{ESO}$[k]$ $\leq$ \emph{INC}$^s[k]$ for any $k\geq 1$.
\end{corollary}

Thus, by Theorem~\ref{the: Translation from ESO to EXC}, $k$-ary inclusion logic with strict semantics is \emph{at least as expressive} as $k$-ary exclusion logic on the level of sentences. Recall that by Corollary~\ref{the: INC & EXC}, with the standard (lax) semantics, INC[$k$] is strictly weaker than EXC[$k$] on the level of sentences. Consequently INC$^s$[$k$] is strictly more expressive than INC[$k$] for any $k\geq 1$\footnote{For the general case this is a known results by \cite{Kontinen13a}. But, to our best knowledge, this is a new result for bounded arity fragments of inclusion logic.}.

Since inclusion logic with strict semantics is equivalent with ESO by \cite{Kontinen13a}, it would be natural to predict that INC$^s[k]$ $\equiv$ ESO[$k$] for any $k\geq 1$. However, for now we only have a lower bound for the expressive power of INC$^s[k]$. To our understanding, a translation from INC$^s[k]$ to ESO[$k$] cannot be achieved by modifying the translation from INC$[k]$ to ESO[$k$] (in \cite{Ronnholm15}) in any straightforward way. Therefore we leave this question as an open problem for further research. 


\subsection{On the relationship between $\ESO$ and various logics with team semantics}

We give here some final remarks on the correspondence between ESO and various logics with team semantics. On the level of sentences the whole ESO can be captured with several logics in this framework, such as dependence logic, independence logic, exclusion logic or inclusion logic with strict semantics. But what are the differences between these approaches and which approach can be considered the most natural or practical?

Usually we do not need the whole ESO and some of its simpler fragments suffice. By restricting the arities of atoms in either dependence or independence logic, we can naturally capture all the functional arity fragments of ESO. However, supposing that the functional arity fragments of ESO differ from the relational ones, then the arity fragments of dependence or independence logic cannot capture any of the relational fragments of ESO -- of which ESO[1] and ESO[2] are particularly natural. These and all the other relational fragments can be captured with the fragments of INEX and EXC.

By examining the actual translations that have been presented, we believe that the compositional translation from ESO[$k$] to INEX[$k$], presented in \cite{Ronnholm15}, is currently the most simple and straightforward. By the results of this paper, we know that inclusion atoms are not needed in order to formulate this translation. However, in order to get rid of inclusion atoms, we had to do several ``tricks'' which made the translation more complicated and unnatural. 
%

So far we have only considered this correspondence on the level of sentences. In order to capture ESO  \emph{on the level of formulas}, we need either independence logic or inclusion-exclusion logic. From the know translations, the one between ESO and INEX (\cite{Ronnholm15}) respects the arity fragments in a natural way. It should also be noted that inclusion atoms are crucial for this translation and we cannot simulate them with exclusion atoms (as we could do on the level of sentences).
By the observations given here, we argue the results of this paper rather complement the results of \cite{Ronnholm15} instead of trivializing them.


\section{Conclusion}\label{sec: Conclusion}

In this paper we have analyzed the expressive power of $k$-ary exclusion atoms. We first observed that the expressive power of EXC[$k$] is between $k$-ary and ($k$$+$$1$)-ary dependence logics, and that when $k=1$, these inclusions are proper. By simulating the use of inclusion atoms with exclusion atoms and by using the complementary values, we were able to translate ESO[$k$]-sentences into EXC[$k$]. By combining this with our earlier translation we managed to capture the $k$-ary fragment of ESO by using only $k$-ary exclusion atoms, which resolves the expressive power of EXC[$k$] on the level of sentences. However, on the level of formulas our results are not yet conclusive.

As mentioned in the introduction, by \cite{Durand12}, on the level of sentences $k$-ary dependence logic captures the fragment of ESO where ($k$$-$$1$)-ary functions can be quantified. Thus $1$-ary dependence logic is not more expressive than FO, but $2$-ary dependence logic is strictly stronger than EMSO -- which can be captured with EXC[$1$]. Also, the question whether EXC[$k$] is properly in between $k$- and ($k$$+$$1$)-ary dependence logic for all $k\geq 2$, amounts to showing whether $k$-ary relational fragment of ESO is properly between ($k$$-$$1$)-ary and $k$-ary functional fragments of $\ESO$ for any $k\geq 2$. To our best knowledge this is still an open problem, even though, by the result of Ajtai \cite{Ajtai83}, both relational and functional fragments of $\ESO$ have a strict arity hierarchy (over arbitrary vocabulary).

In order to formulate the translation in our main theorem, we needed use a new operator to called \emph{unifier} which is expressible in exclusion logic.  This is a very simple but interesting operator for the framework of team semantics by its own right, and its properties deserve to be studied further -- either independently or by adding it to some other logics in this framework.

Finally we used the techniques developed in this paper to formulate a translation from ESO[$k$] to $k$-ary inclusion logic with strict semantics (INC$^s[k]$). We left as an open problem whether INC$^s[k]$ captures ESO[$k$] or is even stronger.


\small
\bibliographystyle{abbrv} 
\bibliography{RR-citations}


\end{document}